\documentclass[pdflatex,sn-mathphys]{sn-jnl}

\usepackage[T1]{fontenc}
\usepackage[utf8]{inputenc}
\usepackage[english]{babel}

\usepackage{graphicx}
\usepackage[protrusion=true,expansion=true]{microtype}

\usepackage{amsfonts,amsmath,amsthm,amssymb}
\usepackage{mathtools}
\usepackage{empheq}
\usepackage{bbm}
\usepackage{mathrsfs}

\usepackage{xcolor}
\usepackage{verbatim}
\usepackage{placeins}
\usepackage{paralist}

\usepackage{algorithm}
\usepackage{algpseudocode}

\usepackage{caption}
\usepackage{subcaption}
\usepackage{appendix}
\usepackage{etoolbox}

\usepackage{hyperref}

\newtheorem{thm}{Theorem}
\newtheorem{lem}[thm]{Lemma}
\newtheorem{cor}[thm]{Corollary}
\newtheorem{prop}[thm]{Proposition}

\theoremstyle{definition}

\newcommand{\bsym}[1]{\boldsymbol{#1}}
\newcommand{\norm}[1]{\lVert#1\rVert}
\newcommand{\dual}[2]{\langle#1,#2\rangle}

\newcommand{\bu}{\bsym{u}}
\newcommand{\bw}{\bsym{w}}

\newcommand{\bv}{\bsym{v}}

\newcommand{\id}{\text{id}}

\newcommand{\Df}{\mathcal{D}(M)}
\newcommand{\Prob}{\mathcal{P}(M)}

\definecolor{Bleu}{rgb}{0.3,0.2,1.0}

\usepackage[T1]{fontenc}
\usepackage{lmodern}
\usepackage{mathtools}
\usepackage{cancel}
\makeatletter
\usepackage{titlesec}
\titleformat{\section}{\fontsize{12}{14}\bfseries}{\thesection}{1em}{}
\titleformat{\subsection}{\fontsize{10}{12}\bfseries}{\thesubsection}{1em}{}
\makeatother

\usepackage{microtype}
\usepackage{hyperref}
\hypersetup{
  colorlinks=true,
  citecolor=green!50!black,   % reference/citation links — light green
  linkcolor= blue!80!white,     % equation/internal links — light red
}
\usepackage{cite}
\usepackage{tikz}
\usetikzlibrary{arrows.meta, calc}

\theoremstyle{definition}

\numberwithin{equation}{section}

\begin{document}

\title[Thermodynamically Constrained IGR] {Thermodynamically Constrained Information Geometric Regularization for Compressible Flows}

\author*[1]{\fnm{Seth} \sur{Taylor}}\email{seth.taylor@usask.ca}
\author[1]{\fnm{Raymond J.} \sur{Spiteri}}
\author[2]{\fnm{St{\'e}phane} \sur{Gaudreault}}

\affil[1]{\orgdiv{Department of Computer Science}, \orgname{University of
    Saskatchewan}, \orgaddress{\city{Saskatoon}, \country{Canada}}}
\affil[2]{\orgdiv{Recherche en prévision numérique atmosphérique}, \orgname{Environnement et Changement climatique Canada}, \orgaddress{\city{Dorval}, \country{Canada}}}

\abstract{We construct and analyze a thermodynamic extension of the recently proposed information geometric regularization of Cao and Sch{\"a}fer. The construction extends their shock-mitigating Hessian metric geometry using the Shannon entropy to constrain the regularized motion based on a thermodynamic length. Reformulating the equations in terms of mass and specific entropy explicitly connects the thermodynamic state to a position in the diffeomorphism group, allowing for a derivation of the regularized equations using an information geometric mechanics formalism based on geodesics on a Hessian manifold with a dual affine connection. The dynamics are defined using a pullback geometry for the Levi--Civita connection, describing constrained geodesic motion, and the cubic Amari--Chentsov tensor describing the information geometric correction. This new compressible fluid model introduces an anisotropic stress tensor to the momentum equation that vanishes along isentropic directions and an additional elliptic equation coupled to the barotropic regularization. Numerical simulations in one and two spatial dimensions demonstrate that the geometrically consistent incorporation of a thermodynamic constraint mitigates cusp singularities previously observed in other approaches while still maintaining the benefits of an inviscid regularization.}

\maketitle
\setcounter{tocdepth}{2}
\vspace{-0.5cm}
\tableofcontents

\section{Introduction}

\noindent
\emph{Information Geometric Regularization for Compressible Flows}
\vspace{0.1cm}

The numerical simulation of fluid dynamic models such as the Euler and Navier--Stokes equations that involve thermodynamic processes is central to the modelling, prediction, and control of natural phenomena \cite{anderson1995computational, pirozzoli2011numerical, bauer2015quiet}. Despite their broad applicability, the full system of equations remains challenging to solve accurately and efficiently. Much of this difficulty stems from dynamical interactions across spatial and temporal scales, which generate multiscale solution structures that are expensive to resolve numerically. Turbulence and shock formation are two of the most prominent consequences. While turbulence remains a central challenge in both pure and applied mathematics, shock formation is a more direct structural property of the equations themselves. Although shocks may serve as an effective model for sharp moving fronts, the formation of a genuine singularity in the macroscopic description of the dynamics of a fluid can be seen as a pathology of the model rather than as a physical feature. \par

Numerical methods designed to handle and capture shocks in compressible fluids have enabled applications such as aerodynamics, inertial confinement fusion, and numerical relativity \cite{pirozzoli2011numerical, lehner2014numerical}. A precursor of many of these schemes was the work of Godunov \cite{godunov1959finite}, which demonstrated that discontinuities could be incorporated into a finite volume method through the solution of local Riemann problems. This approach was followed by foundational work on approximate Riemann solvers \cite{roe1981approximate, harten1983upstream, toro2009hll, toro2013riemann}, higher-order accurate methods via reconstruction techniques such as monotonicity upstream methods \cite{van1979towards}, essentially non-oscillatory \cite{harten1997uniformly}, and weighted essentially non-oscillatory methods \cite{liu1994weighted, jiang1996efficient}, and nonlinear reconstruction schemes avoiding the use of Riemann solvers \cite{nessyahu1990non, kurganov2000new}. These methods all share a common starting point: they solve the original system, which admits singularities, and regularize the problem numerically. An alternative standpoint is to modify the model into a new system of equations that maintains the physics of sharp moving fronts without a mathematical singularity formation. Viscous regularizations \cite{vonneumann1950method, persson2006sub, boyd2001chebyshev, bhagatwala2009modified}, which systematically introduce dissipation to create a smoothing effect on the shock formation, are the most common approach along this line. Regularization via dissipation, however, fundamentally alters the equations of motion and can compromise their turbulent multi-scale dynamics by removing energy from the system at small length scales \cite{wilkins1980use, frisch2008hyperviscosity, cook2005hyperviscosity}. \par

The recently proposed information geometric regularization (IGR) of Cao and Sch{\"a}fer \cite{cao2023information} offers a novel method for \emph{inviscid regularization} by leveraging a geometric relation between shock formation and a boundary geometry on the space of diffeomorphisms of the fluid domain \cite{khesin2007shock}. The loss in regularity exhibited by the barotropic Euler equations can be associated with the degeneration of the Jacobian determinant of the Lagrangian flow map. In a beautiful application of geometry to scientific computing, the IGR methods leveraged an interior-point formulation to regularize the equations of motion based on a convex barrier functional that extends the vanishing Jacobian determinant boundary to infinity. Convex functionals of this form and their resulting Hessian metrics are common tools in information geometry \cite{amari2016information} on the space of probability densities. The methods introduced by \cite{cao2023information} provide a first-principles hydrodynamical approach to mitigating shock formation by introducing an alternative mechanical structure on the space that differs from the Lagrangian or Hamiltonian structures commonly used to define alternative flow models. It has been demonstrated that, in one spatial dimension, these equations are globally well-posed \cite{cao2024information}. \\ 

\noindent
\emph{Thermodynamic Extensions}
\vspace{0.1cm}

The IGR methods were originally derived in a barotropic setting, and first steps have been taken in the work of \cite{wilfong2025simulating} to incorporate thermodynamic processes where the regularized equations of motion took the form 
\begin{equation}
\label{full_compressible_euler_regularized}
\begin{aligned}
\partial_t \rho \bu + \nabla \cdot(\rho \bu \otimes \bu + (p(E,\rho) + \Sigma) I) & = 0 \,,
\\
\partial_t \rho + \nabla \cdot (\rho \bu) &= 0 \,,
\\
\partial_t E + \nabla \cdot ((E + p + \Sigma)\bu) & = 0 \,,
\\
\rho^{-1} \Sigma - \alpha \nabla \cdot (\rho^{-1} \nabla \Sigma) &= \alpha \left((\mathrm{div}(\bu))^2 + \mathrm{tr}(\nabla \bu^2)\right) \,,
\end{aligned}
\end{equation}
where $\bu$ is the fluid velocity, $\rho$ the mass density, $p$ is the pressure, $E$ is the total energy density, and $\Sigma$ is an `entropic pressure' potential associated to the IGR regularization. This regularization approach was an enabling factor to circumvent the computational overhead of shock capturing for the simulation of multi-engine spacecraft at unprecedented scales \cite{wilfong2025simulating}. However, the regularized system \eqref{full_compressible_euler_regularized} was later shown to exhibit a thermodynamic inconsistency, manifested by a spurious spike in internal energy at wave collisions \cite{barham2025hamiltonian}. Barham et al. attributed this inconsistency to a geometric mismatch between the barotropic regularization and the full compressible Euler system and introduced several Hamiltonian regularizations of the equations of motions in response. Although these new sets of equations exhibited a conservation of energy, there still remained a spurious spike in energy density at wave-front collisions and they were recognized to introduce additional complications, such as higher-order derivatives in the state variables, that would hinder their numerical use. Their work highlighted that the information geometric mechanics fundamentally differs from the variational structure of Hamiltonian mechanics, and the resolution of this thermodynamic inconsistency within this framework was left as an open problem \cite{barham2025hamiltonian}. \\

\noindent
\emph{Thermodynamically Constrained Information Geometric Regularization}
\vspace{0.1cm}

The aim of this work is to develop an extension of the IGR methods that incorporates a thermodynamic regularization directly into the Hessian metric geometry defining the dynamics. The starting point is to rewrite the thermodynamic balance relations in terms of the entropy density $\pi = \rho s$, where $s$ is the specific entropy, which satisfies a conservation law of the form
\begin{equation}
\label{pi_continuity}
\partial_t \pi  + \nabla \cdot (\pi \bu) = 0 \,.
\end{equation}
As a transported quantity, the entropy density evolves via pushforward $\pi = \varphi_* \pi_0$ with the Lagrangian flow map $\varphi(t):M \to M$ defined by $\dot\varphi = \bu \circ \varphi$. This perspective offers connection between the full thermodynamic state and a position within the diffeomorphism group. We extend the approach taken by Cao and Sch{\"a}fer by defining the equations of motion using a convex barrier function in an ambient space in which the diffeomorphism group is embedded. In this setting, we consider an embedding 
\begin{equation*}
\Df \hookrightarrow \Df \times \Prob \times \mathcal{V}(M) \,,
\end{equation*}
which encodes the fluid deformation in the diffeomorphism group $\Df$ of the fluid domain $M$, the mass density in the space of probabilities $\Prob$, and the entropy density in the space of positive densities (volume forms) $\mathcal{V}(M)$. A similar submanifold geometry was used by Cao and Schafer \cite{cao2023information} to derive the barotropic information geometric regularization with the first two components. We extend their convex barrier by constraining the motion to take a path that minimizes a \emph{thermodynamic length} associated to the Kullback--Leibler divergence
\begin{equation}
\label{KL_divergence}
\mathscr{D}_{KL}(\pi \mid\mid\bar{\pi}) = \int \pi(x) \log(\pi(x)/\bar{\pi}(x))dx\,,
\end{equation}
where $\bar{\pi}$ is a reference entropy density. The equations of motion are then derived using information geometric mechanics, based on the dual affine connection of the pullback Hessian metric defined by the barrier function and the embedding. The resulting thermodynamically constrained information geometric regularization of the Euler (TIGRE) equations take the form
\begin{equation}
\label{eq:TIGRE}
\begin{aligned}
\partial_t \rho \bu + \nabla \cdot \left(\rho \bu \otimes \bu  + (p(\rho,s) + \Sigma)I\right) &= -\pi \nabla \chi\,,
\\
\partial_t \rho + \nabla \cdot (\rho \bu) &= 0 \,,
\\
\partial_t \pi + \nabla \cdot (\pi \bu) &= 0 \,,
\\
\rho^{-1}\Sigma - \alpha \nabla \cdot \left(\rho^{-1}(\nabla \Sigma + \pi \nabla \chi)\right) &= \alpha \left(\mathrm{div}(\bu)^2 + \mathrm{tr}(\nabla \bu^2)\right)\,,
\\
\chi - \beta \pi^{-1} \nabla \cdot\left(\pi \rho^{-1}(\nabla \Sigma + \pi \nabla \chi)\right) &=  \beta \big(\mathrm{div}_\pi(\bu)^2 - \mathrm{tr}(\nabla \bu^2) + \nabla^2\log\pi[\bu, \bu]\big)\,,
\end{aligned}
\end{equation}
where $\mathrm{div}_\pi(\bu) = \pi^{-1}\mathrm{div}(\pi \bu)$ is an entropy weighted divergence and the additional entropic pressure potential $\chi$ arises from the extended barrier function, with the constants $\alpha, \beta > 0$ controlling the strength of the regularization. The TIGRE system reduces to the barotropic IGR regularization as $\beta \to 0$ and possesses many of the same desirable properties of the inviscid regularization. The additional potential responds to entropy weighted deformations of the flow by concentrating or expanding to generate a curvature-driven restoring force. The regularization can be naturally applied to the Navier--Stokes equations by including viscosity and forcing terms; however, its geometric structure is most transparently presented in the inviscid setting. We utilize an information geometric formalism to provide a rigorous derivation of the system \eqref{eq:TIGRE} and further validate the use of this constraint numerically by demonstrating the flow mitigates the cusp-singularities discussed in \cite{barham2025hamiltonian}.

\noindent
\emph{Outline} 
\vspace{0.1cm}

In Section~\ref{sec:geometric_formulation}, we provide a formalism for the information geometric mechanics used to define the TIGRE equations \eqref{eq:TIGRE} and demonstrate that it reproduces the IGR equations \cite{cao2023information}. In Section~\ref{sec:entropy_barrier}, we then apply the formalism to our proposed thermodynamic extension to derive the equations of motion \eqref{eq:TIGRE} and compare their thermodynamic, conservative, and dynamic properties to the IGR system \eqref{full_compressible_euler_regularized}. In Section~\ref{sec:numerical_experiments}, we perform numerical experiments that validate our theoretical considerations in comparison to the barotropic IGR regularization method. We conclude with some proposed future directions of research into information geometric mechanics at the intersection of geometric mechanics, information theory, and scientific computing.  

\section{Geometric Formalism}
\label{sec:geometric_formulation}

We begin by recalling the Riemannian geometric formalism of Khesin, Misio{\l}ek, and Modin \cite{khesin2021geometric} for the description of the full compressible Euler equations as a form of Newton's equations on the diffeomorphism group. We then discuss the dual affine connections and the pullback geometry needed to define the information geometric regularization. This information geometric mechanics is then applied to derive the IGR regularization \cite{cao2023information}, which offers an alternative derivation that we describe in detail because the same approach is applied analogously to the derivation of the TIGRE system \eqref{eq:TIGRE} in Section~\ref{sec:entropy_barrier}.  

\subsection{Newton's Equations on the Diffeomorphism Group}

Let $M \subset \mathbb{R}^d$ denote an embedded submanifold representing the fluid domain, and denote $\mu$ as the restriction of the ambient volume form to $M$. The compressible Euler equations can be viewed as Newtonian dynamics on the diffeomorphism group $\Df$ equipped with a Riemannian metric \cite{khesin2021geometric}. The tangent spaces are identified with the space of vector fields on $M$, denoted $\mathfrak{X}(M)$, in the sense that
\begin{equation}
  T_\varphi \Df = \left\{\bv: M \to TM \,:\, \pi_{TM} \circ \bv = \varphi\right\} = \mathfrak{X}(M) \circ \varphi \,,
\end{equation}
where $\pi_{TM}:TM \to M$ is the tangent bundle projection. When $\partial M \neq \emptyset$, we restrict to vector fields that are tangent to the boundary satisfying $v \cdot n\vert_{\partial M}= 0$. The fluid configuration is described by the Lagrangian flow map $\varphi(t) \in \Df$ and the Eulerian fluid velocity $u(t) \in \mathfrak{X}(M)$ is defined by $\dot\varphi = u \circ \varphi$. The fluid mass density is represented by a positive measure $\varrho = \rho\mu$ associated with an element of the space of probability densities
\begin{equation}
\label{prob_densities}
\Prob = \left\{\rho \in C^{\infty}(M) \,:\, \rho > 0 \,,\, \int_M \rho \mu = \int_M \mu \right\} \,.
\end{equation}
The diffeomorphism group acts on \eqref{prob_densities} by the right action of pullback
\begin{equation}
\label{pullback_action}
\varphi^*\varrho = \rho \circ \varphi J_{\mu}(\varphi) \mu \,,
\end{equation}
where $J_\mu(\varphi) \in C^{\infty}(M)$ is the Jacobian determinant of the transformation, and on the left by the pushforward action 
\begin{equation}
\label{pushforward_action}
\varphi_* \varrho \coloneqq (\varphi^{-1})^* \varrho \,.
\end{equation}
The metric relevant to the compressible Euler equations is the $L^2$ metric, expressed in the Lagrangian and Eulerian frames as  
\begin{equation}
\label{L2_metric}
G_\varphi(\dot{\varphi}, \dot\varphi) = \int_M |\dot\varphi|^2 \mu = \int_M \rho |\bu|^2 \mu \,.
\end{equation}
Potential energies on $\Df$ are defined by the pullback of a potential $\bar{U}: \Prob \to \mathbb R$ as $U(\varphi) = \bar{U}(\varphi_*\mu)$. Using this construction, Newton's equations on the diffeomorphism group are the second-order differential equation of the form \cite{khesin2021geometric}
\begin{equation}
\label{Newtons_eqs}
\nabla^G_{\dot\varphi} \dot\varphi = - \nabla^G U(\varphi)\,,
\end{equation}
where $\nabla^G$ is the covariant derivative defined by the metric \eqref{L2_metric}. The covariant derivative on $\Df$ is lifted pointwise from the metric covariant derivative on $M$, and by computing the gradient of the
pullback potential energy using an infinite-dimensional analogue of the divergence theorem (see \cite{khesin2021geometric} Lemma 2.3), the Newton equations \eqref{Newtons_eqs} admit the equivalent expression
\begin{equation}
\label{Newton_L2_metric_gradient}
\nabla_{\dot \varphi} \dot \varphi = - \bigg(\nabla \frac{\delta \bar{U}}{\delta \rho} \bigg)\circ \varphi \,,
\end{equation}
where $\nabla$ is the gradient defined by the metric on $M$ and the covariant derivative $\nabla_{\dot\varphi}\dot\varphi$ is associated to the material derivative of the Eulerian velocity field via
\begin{equation}
\nabla_{\dot\varphi} \dot\varphi = \left(\frac{D \bu}{Dt}\right) \circ \varphi = \big(\partial_t \bu + \nabla_{\bu}\bu\big) \circ \varphi \,,
\end{equation}
where the coordinate expression of the covariant derivative is given by $(\nabla_{\bu}\bu)^k = u^i \partial_i u^k + \Gamma^k_{ij} u^i u^j$ where $\Gamma^i_{jk}$ are the Christoffel coefficients. The Newton system can be extended to incorporate other thermodynamic processes using a more general pullback potential energy \cite{khesin2021geometric}.\\

\noindent
\emph{Adiabatic Compressible Flow}
\vspace{0.1cm}

An adiabatic fluid possesses an extended thermodynamic state involving temperature and mass density, remaining isolated within fluid parcels without external heat exchange, doing work only through changes in pressure \cite{dolzhansky2012fundamentals}. It is convenient in the context of the variational formulation to work with the entropy of the fluid, and we consider the thermodynamic state as an element $(\rho, \pi) \in \Prob \times \mathcal{V}(M)$, where 
\begin{equation}
\label{volume_forms}
\mathcal{V}(M) = \left\{\nu \in \Omega^d(M) \,:\, \nu > 0  \right\} 
\end{equation}
is the space of volume forms. The entropy density $\pi \in \mathcal{V}(M)$ can be written as $\pi = \rho s \mu$ where $s \in C_+^{\infty}(M)$ is the specific entropy and the fluid potential energy becomes
\begin{equation}
\bar{U}(\rho,s) = \int_M \rho e(\rho,s) \mu \,,
\end{equation}
where $e(\rho,s)$ is the internal energy density. In an adiabatic fluid, the second law of thermodynamics 
\begin{equation}
\vartheta(\rho,s) ds \geq \delta Q \geq 0\,,
\end{equation}
where $\vartheta(\rho,s)$ is the temperature and $\delta Q$ is the change in heat, holds with equality $\vartheta ds = \delta Q$. This reversibility of the flow gives the continuity equation \eqref{pi_continuity} for the entropy density and in turn implies that the specific entropy is advected
\begin{equation}
\partial_t s + u \cdot \nabla s = 0 \implies s(t) = s_0 \circ \varphi^{-1}\,,
\end{equation}
using the continuity equation for the mass density. The first law of thermodynamics $dU = \delta Q - \delta W$ can then be expressed as the Gibbs relation
\begin{equation}
\label{first_law}
de(\rho,s) =  \vartheta(\rho,s)ds + p(\rho,s)d(\rho^{-1})\,,
\end{equation}
where the temperature and pressure are defined by the partial differential relations
\begin{equation}
\label{thermodynamic_quantities}
\vartheta(\rho,s) = \partial_s e(\rho,s) \,, \qquad p(\rho,s) = \rho^{2} \partial_\rho e(\rho,s) \,.
\end{equation} 
The potential on the diffeomorphism group becomes
\begin{equation}
\label{full_pullback_potential}
U(\varphi) = \bar{U}(\varphi_* \mu, s_0 \circ \varphi^{-1}) \,,
\end{equation} 
and the Newton's equations are then defined by its metric gradient.
\begin{lem}
The gradient of the potential \eqref{full_pullback_potential} with respect to the $L^2$ metric \eqref{L2_metric} is given by
\begin{equation}
\nabla^GU(\varphi) = \left(\nabla h(\rho,s) - \vartheta(\rho,s)\nabla s\right) \circ \varphi\,,
\end{equation}
where $h(\rho,s) = e(\rho,s) + \rho^{-1}p(\rho,s)$ is the specific enthalpy.
\end{lem}

\makeatletter
\renewenvironment{proof}[1][\proofname]{%
  \par
  \pushQED{\qed}%
  \normalfont\normalsize
  \topsep6\p@\@plus6\p@\relax
  \trivlist
  \item[\hskip\labelsep\itshape #1\@addpunct{.}]%
  \ignorespaces
}{%
  \popQED\endtrivlist\@endpefalse
}
\makeatother

\begin{proof}
Let $\varphi(t) \in \Df$ be a path with $\rho(t)\mu = \varphi_*\mu$ and $s(t) = s_0 \circ \varphi^{-1}$. By definition of the Riemannian gradient, we can write   
\begin{equation*}
G_\varphi(\nabla^G U(\varphi), \dot\varphi)  = \frac{d}{dt} U(\varphi) = \frac{d}{dt} \bar U(\rho,s) \,,
\end{equation*}
from which it follows that 
\begin{equation*}
\begin{aligned}
\frac{d}{dt}\bar{U}(\rho,s) &= \int_M \left(\rho\partial_s e(\rho,s) \dot{s}  + (e(\rho,s) + \rho \partial_\rho e(\rho,s))\dot{\rho}\right) \mu
\\
&= \int_M \left(-\rho \vartheta(\rho,s)g(\bu,\nabla s) - h(\rho,s)\nabla \cdot (\bu \rho)\right) \mu
\\
& = \int_M \rho g(\nabla h - \vartheta \nabla s, \bu) \mu = G_\varphi(\left(\nabla h(\rho,s) - \vartheta(\rho,s)\nabla s\right) \circ \varphi, \dot{\varphi})\,,
\end{aligned}
\end{equation*}
integrating by parts and using the definition of the metric.
\end{proof}

\begin{cor}
The Eulerian form of the Newton's equations \eqref{Newtons_eqs} with potential \eqref{full_pullback_potential} is equivalent to the compressible Euler equations
\begin{equation}
\label{full_compressible_euler}
\begin{aligned}
\partial_t \rho \bu + \nabla \cdot(\rho \bu \otimes \bu +p(s,\rho) I) & = 0 \,,
\\
\partial_t \rho + \nabla \cdot (\rho \bu) &= 0 \,,
\\
\partial_t \pi + \nabla \cdot(\pi \bu) & = 0 \,,
\end{aligned}
\end{equation}
with $\dot\varphi = \bu \circ \varphi$, $\rho \mu = \varphi_* \mu$, and $\pi = \rho s$ where $s(t) = s_0 \circ \varphi^{-1}$.
\end{cor}
\begin{proof}
The pressure definition \eqref{thermodynamic_quantities} relates to the temperature and enthalpy via
\begin{equation}
\frac{1}{\rho} \nabla p = \nabla h - \vartheta \nabla s \,,
\end{equation}
which allows us to write the non-conservative form of the equations by composing the Newton system with the inverse $\varphi^{-1}$ to give
\begin{equation}
\partial_t \bu + \nabla_{\bu} \bu = -\frac{1}{\rho}\nabla p\,.
\end{equation}
Multiplying by the mass density and using the conservation of mass and the identity
\begin{equation*}
\nabla \cdot (\rho \bu \otimes \bu) = \rho \nabla_{\bu} \bu +  \bu \nabla \cdot (\rho \bu) 
\end{equation*}
then gives the conservative form of the momentum equation. The continuity equations are defined directly from the constraints $(\pi,\varrho) = \varphi_*(\pi_0, \rho_0\mu)$ defining the pullback potential. 
\end{proof}

The full compressible Euler equations are commonly expressed in a flux form with the state variable $(\rho \bu, \rho, E)$, where $E$ is the total energy density
\begin{equation}
\label{total_energy_density}
E = \rho\big( e(\rho,s) + \frac{1}{2}|\bu|^2\big) \,.
\end{equation}
Given an internal energy density $e = e(\rho,s)$, strict positivity of the temperature $\vartheta(\rho,s) = \partial_s e(\rho,s) > 0$ ensures that $s \mapsto e(\rho,s)$ is invertible. This yields a relation $s = s(\rho,e)$ and an expression of the pressure as
\begin{equation}
p(\rho,E) = p\big(\rho,s\big(\rho,E/\rho - \frac{1}{2}|\bu|^2\big)\big)\,,
\end{equation}
and an equivalent form of the equations of motion \eqref{full_compressible_euler} given by
\begin{equation}
\label{full_compressible_euler}
\begin{aligned}
\partial_t \rho \bu + \nabla \cdot(\rho \bu \otimes \bu +p(E,\rho) I) & = 0 \,,
\\
\partial_t \rho + \nabla \cdot (\rho \bu) &= 0 \,,
\\
\partial_t E + \nabla \cdot((E + p) \bu) & = 0 \,.
\end{aligned}
\end{equation}
In an ideal gas, the internal energy density is given by
\begin{equation}
\label{ideal_gas_law}
e(\rho, s) = \frac{c_v}{\gamma -1}\rho^{\gamma-1}e^{s/c_v}\,,
\end{equation}
where $c_v$ is the specific heat capacity at constant volume and $\gamma >1$ is the adiabatic ratio $\gamma = c_v/c_p$. This gives the relation $p = (\gamma -1)\rho e$, from which we can write the pressure in the equivalent forms
\begin{equation}
\label{total_energy_ideal_gas}
p(E,\rho) = (\gamma - 1) \left(E - \frac{1}{2}\rho|\bu|^2\right), \qquad p(\rho,s) = \rho^{\gamma}e^{s/c_v} \,.
\end{equation}
The pressure relations \eqref{total_energy_ideal_gas} will be used to consistently relate the thermodynamic extension of the IGR equations \eqref{full_compressible_euler_regularized}, written in terms of the dynamical variables $(\rho, \rho u, E)$, to the TIGRE system \eqref{eq:TIGRE} written in terms of $(\rho, \rho u, s)$. The TIGRE system uses the transported thermodynamic variables $(\rho,s)$ as they can be related explicitly to the state $\varphi(t) \in \Df$ without reference to the dynamics. This was an enabling step to define the thermodynamically constrained information geometric regularization of the system \eqref{full_compressible_euler}. Before elaborating on this construction further, we first recall the basic geometric structures of information geometry used to define the regularization.

\subsection{Hessian Manifolds}

Information geometry (IG) relates to the study of differential geometric structures on the space of probability densities and measures \cite{ay2017information, amari2016information}. Originating from Rao \cite{rao1945information} and Chentsov's \cite{chentsov1982statistical} studies on Fisher information and culminating in Amari's seminal differential geometric theory of information \cite{amari1980theory, amari1982differential}, IG is now a mature field of study with a host of applications in statistics, machine learning, artificial intelligence, and neuroscience \cite{amari2016information}. While the classical notions in IG are formulated for finite-dimensional parametric families of probabilities (statistical manifolds) \cite{lauritzen1987statistical}, infinite-dimensional extensions of the theory have been devised \cite{pistone1995infinite, ay2017information, khesin2024information, bauer2024p}, and our work is formally in a smooth setting. The information geometric mechanics used to derive the regularized compressible flow equations differs from the Riemannian metric geometry defining the Newton's equations \eqref{Newtons_eqs}. Here, we recall the basic structures needed to describe these mechanical connections and their constrained form to an embedded submanifold of diffeomorphisms. \\

\noindent
\emph{Dual Geodesic Equation}
\vspace{0.1cm}

Let $\mathcal{E}$ be a (possibly infinite-dimensional) smooth manifold and $\psi: \mathcal{E} \to \mathbb{R}$ a strictly convex functional. The second Fr{\'e}chet derivative of this potential defines a metric of the form $G^{\psi} = D^2 \psi$, and the pair $(\mathcal{E}, G^{\psi})$ forms a Hessian manifold \cite{shima1997geometry}. Fundamental to these spaces is the dual pair of affine connections $(\nabla, \nabla^*)$. The so-called primal connection $\nabla$ is defined by
\begin{equation*}
G^{\psi}_\nu(X,Y) = X(Y\psi) - (\nabla_X Y)\psi\,,
\end{equation*}
for all $X,Y \in T_\nu \mathcal{E}$, and the dual connection $\nabla^*$ is defined via the relation
\begin{equation*}
X[G^{\psi}(Y,Z)] = G^{\psi}(\nabla_X Y, Z) + G^{\psi}(Y, \nabla^*_X Z) \,.
\end{equation*}
This pair symmetrizes to give the Levi--Civita connection of the metric
\begin{equation}
\nabla^{LC} = \frac{1}{2}( \nabla + \nabla^*) \,.
\end{equation}
Their difference is encoded by the Amari--Chentsov tensor, formed by the symmetric cubic tensor
\begin{equation}
\label{AC_tensor}
C(X,Y,Z) = \left(\nabla_X g\right)(Y,Z) \,.
\end{equation}
In particular, the $(1,2)$-tensor $C^{\sharp}$ defined by $G^\psi((C^{\sharp})_X Y, Z) = C(X,Y,Z)$ gives the identity
\begin{equation}
\label{fundamental_relation}
\nabla^*_X Y = \nabla_X Y + (C^\sharp)_X Y \,.
\end{equation}
The equations of information geometric mechanics can be associated to paths $\eta(t) \in \mathcal{E}$ satisfying the dual geodesic equation
\begin{equation*}
\nabla^*_{\dot{\eta}} \dot\eta = \nabla_{\dot\eta} \dot\eta + (C^{\sharp})_{\dot{\eta}} \dot\eta = 0  \,,
\end{equation*}
which defines a notion of straight lines with respect to the dual affine connection. Using the identity \eqref{fundamental_relation}, the dual geodesic equation can be expressed equivalently as 
\begin{equation}
\label{dual_geodesic_w_LC}
\nabla^*_{\dot\eta}\dot\eta = \nabla^{LC}_{\dot\eta} \dot\eta + \frac{1}{2}\big(C^{\sharp}\big)_{\dot\eta} \dot\eta = 0  \,,
\end{equation}
conveniently separating the Riemannian contribution associated to the kinetic energy minimizing motions and the corrections associated to the information geometric regularization.

\subsection{Pullback of the Dual Geodesics}
\label{sec:AC_computations}

The equations \eqref{full_compressible_euler} are only locally well-posed, and solutions $\varphi(t) \in \Df$ satisfying \eqref{Newtons_eqs} can exit $\Df$ into the space of maps $C^\infty(M,M)$ through the `boundary' of the diffeomorphism group
\begin{equation}
\label{boundary_Diff}
\partial\Df = \left\{F \in C^{\infty}(M,M) \,:\, J_{\mu}(F) = 0 \right\} \,,
\end{equation}
in finite time. The space \eqref{boundary_Diff} describes non-invertible transformations of the fluid domain with singular Jacobian determinant expressing the onset of singularities due to shock formation. A geometric intuition behind IGR is to move this boundary to be infinitely far away by introducing a convex barrier functional that diverges as $J_\mu(\varphi) \to 0$. Particle collisions can then only occur asymptotically as $t \to \infty$, resulting in dynamics that remain within the diffeomorphism group \cite{cao2023information, cao2024information}. \par
The non-convexity of the mapping $\varphi \mapsto J_\mu(\varphi)$ obstructs the definition of a Hessian metric directly on $\Df$. Cao and Sch{\"a}fer handled this obstruction by defining the barrier functional in the ambient space $\Df \times \Prob$ and use an embedding of the form
\begin{equation}
\label{embedding}
\iota: \Df \to \Df \times \Prob \,, \qquad \iota(\varphi) = (\varphi, \varphi^*\mu) \,.
\end{equation}
In the ambient space, the non-convexity could be realized as a submanifold constraint and the equations of motion defined by restriction of the dual geodesic equation on the submanifold 
\begin{equation*}
\mathcal{M} = \iota(\Df) \subset \Df \times \Prob \,.
\end{equation*} 
The differential of \eqref{embedding} is a linear isomorphism of the tangent spaces $T_\varphi \Df \simeq T_{\iota(\varphi)}\mathcal{M}$, where the latter can be defined by the image of the projection 
\begin{equation}
\Pi_\eta: T_\eta\mathcal{E} \to T_\eta \mathcal{M}\,,
\end{equation}
at $\eta = \iota(\varphi)$. The embedding induces a pullback metric and Amari--Chentsov defined, respectively, by
\begin{equation}
\label{pullback_objects}
\begin{aligned}
\tilde{G}_{\varphi}^{\psi}(U,V) &\coloneqq (\iota ^* G^{\psi})(U,V) = G^\psi(D\iota_\varphi(U), D\iota_\varphi(V))\,,
\\
\tilde{C}_{\varphi}(U,V,W) &\coloneqq (\iota ^* C)(U,V,W) = C(D\iota_\varphi(U), D\iota_\varphi(V), D\iota_\varphi(W))\,,
\end{aligned}
\end{equation}
for all $U,V,W \in T_\varphi\Df$. The analogue of the tensor $C^{\sharp}$ is defined using the pullback metric as 
\begin{equation}
\label{index_raised_pullback_tensor}
\tilde G^{\psi}_\varphi(\tilde{C}^\sharp_U(V), W) = \tilde{C}_{\varphi}(U,V,W)\,.
\end{equation}
Using these pullback objects on $\Df$, we define the analogue dual affine connection as
\begin{equation}
\label{pullback_dual_covariant_deriv}
\tilde{\nabla}^* = \tilde{\nabla}^{LC} + \frac{1}{2}\tilde{C}^{\sharp} \,,
\end{equation}
where $\tilde{\nabla}^{LC}$ is the Levi--Civita connection defined by the pullback metric \eqref{pullback_objects}. A curve $\varphi(t) \in \Df$ is said to satisfy the pullback dual geodesic equation if 
\begin{equation}
\label{pullback_dual_geodesics}
\tilde\nabla^*_{\dot\varphi} \dot\varphi = \tilde\nabla^{LC}_{\dot\varphi}\dot\varphi + \frac{1}{2}\tilde{C}^{\sharp}_{\dot{\varphi}}(\dot\varphi) = 0 \,. 
\end{equation}
We have the following equivalence to the projected dual geodesic equation. 
\begin{lem}
\label{lem:projected_geodesics}
If the curve $\eta(t) = \iota(\varphi(t))$ satisfies the restricted dual geodesic equation 
\begin{equation}
\label{restricted_geodesic}
\Pi_\eta(\nabla^*_{\dot\eta}\dot\eta) = 0\,,
\end{equation}
then $\varphi(t) \in \Df$ satisfies the pullback dual geodesic equation \eqref{pullback_dual_geodesics}.
\end{lem}
\begin{proof}
We can consider the expression for the equation \eqref{restricted_geodesic} in terms of $\varphi(t)$,
\begin{equation}
\Pi_\eta(\nabla^*_{\dot{\eta}} \dot\eta) = \Pi_{\iota(\varphi)}\left(\nabla^{LC}_{D\iota(\dot\varphi)}D\iota\dot\varphi + \frac{1}{2}C^{\sharp}_{D\iota(\dot{\varphi})}(D\iota \dot\varphi) \right) \,,
\end{equation}
and simplify each component individually. Using Gauss' formula for the restriction of the covariant derivative and the isometry $\iota: (\Df,\tilde{G}^{\psi}) \to (\mathcal{M},G^{\psi})$, which holds by construction, we get that
\begin{equation}
\label{first_term_covariant_derivative}
D\iota_\varphi\left( \tilde\nabla^{LC}_X Y \right) = \Pi_{\iota(\varphi)}\left(\nabla^{LC}_{D\iota(\dot\varphi)}D\iota(\dot\varphi)  \right) \,.
\end{equation}
The definition of the ambient operator $C^{\sharp}$ then gives us
\begin{equation*}
G^{\psi}_{\iota(\varphi)}(C^{\sharp}_{D\iota(\dot\varphi)}D\iota(\dot\varphi), D\iota(X)) 
= \tilde{C}_{\varphi}(\dot\varphi, \dot\varphi, X) 
= \tilde{G}^\psi_\varphi(\tilde{C}^\sharp_{\dot{\varphi}} \dot\varphi,X) = G^\psi_{\iota(\varphi)}(D\iota_\varphi( \tilde{C}^\sharp_{\dot{\varphi}} \dot\varphi), D\iota(X))\,,
\end{equation*}
using the isometry in the last equality. Because $D\iota(T_\varphi\Df) = T_\eta\mathcal{M}$ and $\Pi_\eta$ is the $G^\psi$-orthogonal projector onto the tangent space of the submanifold, it follows that
\begin{equation}
\label{second_term_AC_tensor}
\Pi_\eta\left(C^{\sharp}_{D\iota(X)}D\iota(Y)\right) = D\iota_\varphi\left(\tilde{C}^{\sharp}_X Y\right) \,.
\end{equation}
Combining the identities \eqref{first_term_covariant_derivative} and \eqref{second_term_AC_tensor}, we get that 
\begin{equation*}
\Pi_\eta\left(\nabla^*_{\dot\eta}\dot\eta\right) = D\iota_{\varphi}\left( \tilde\nabla^{LC}_{\dot\varphi}\dot\varphi + \frac{1}{2}\tilde{C}^{\sharp}_{\dot{\varphi}}(\dot\varphi) \right) = 0 \,,
\end{equation*}
which implies that $\varphi(t)$ satisfies \eqref{pullback_dual_geodesics} because $D\iota_\varphi: T_\varphi\Df \to T_\eta\mathcal{M}$ is a linear isomorphism.
\end{proof}
Lemma~\ref{lem:projected_geodesics} establishes a convenient formula to derive the regularized equations of motion in the form \eqref{pullback_dual_covariant_deriv}, which we now verify for the pressureless IGR equations.

\subsection{Information Geometric Regularization}
\label{sec:IGR_barotropic} 
In this section, we demonstrate that the pressureless IGR equations are equivalent to a pullback geodesic equation of the form \eqref{pullback_dual_geodesics}. This formulation provides an explicit separation of the two entropic pressures contributing to $\Sigma$, also discussed in~\cite{barham2025hamiltonian}, resulting from a curvature-induced term in the pullback Levi--Civita connection and the barrier geometry encoded by the Amari--Chentsov tensor.\par

The IGR equations proposed by \cite{cao2023information} were formulated using a convex barrier of the form
\begin{equation}
\label{convex_barrier}
\psi: \Df \times \Prob \to \mathbb{R} \,, \qquad \psi(\varphi,\nu) = \frac{1}{2} \int_M |\varphi(x)|^2 d\mu(x) - \alpha \int_M \log(\nu(x))d\mu(x) \,,
\end{equation}
involving a transport cost and a log-barrier function penalizing the degeneration of the Jacobian determinant with the constant $\alpha > 0$ controls the strength of the regularization. The Hessian of the barrier \eqref{convex_barrier} is given by
\begin{equation}
\label{Hessian_barrier}
D^2\psi_{\varphi,\nu}[(U,\eta),(V,\xi)] = \int_M U \cdot V \mu + \alpha \int_M \frac{\eta \xi}{\nu^2} \mu \,,
\end{equation}
which pulls back to a divergence-weighted metric on $\Df$ \cite{barham2025hamiltonian}.
\begin{lem}
\label{lemma:pullback_metric_barotropic}
The pullback Hessian metric defined by the embedding \eqref{embedding} and barrier \eqref{convex_barrier} is given by
\begin{equation}
\label{diverence_weighted_metric}
G^{\psi}_\varphi(U,V) =  \int_M \rho(\bu \cdot \bv  + \alpha (\nabla \cdot \bv) (\nabla \cdot \bu)) \mu\,,
\end{equation}
where $U, V \in T_\varphi \Df$ with $U =  \bu  \circ \varphi$ and $V = \bv \circ \varphi$. 
\end{lem}
\begin{proof}Letting $\gamma(t) \in \Df$ be such that $\gamma(0) = \varphi$ and $\dot{\gamma}(0) = U = \bu \circ \varphi \in T_\varphi\Df$, we see that  
\begin{equation}
D\iota\vert_{\varphi}[U] = \left.\frac{d}{dt}\right\vert_{t=0}(\gamma(t),\gamma^*\mu) = (U, (\nabla \cdot\bu) \circ \varphi \cdot \varphi^*\mu)\,.
\end{equation}
Using the definition of the pullback Hessian metric
\begin{equation}
G^{\psi}_\varphi(U,V) = \left(\iota^* D^2\psi[U,V]\right)(\varphi) = D^2 \psi\vert_{\iota(\varphi)}\left[D\iota_\varphi[U], D\iota_\varphi [V]\right] \,,
\end{equation}
the expression \eqref{Hessian_barrier}, and a change of variables with the inverse map $\varphi^{-1}$, we see that
\begin{equation}
\begin{aligned}
G^{\psi}_\varphi(U,V) &=  \int_M (U \cdot V + \alpha ((\nabla \cdot \bv) (\nabla \cdot \bu)) \circ \varphi) \mu = \int_M \rho(\bu \cdot \bv  + \alpha (\nabla \cdot \bv) (\nabla \cdot \bu)) \mu\,,
\end{aligned}
\end{equation}
which establishes the claim. 
\end{proof}
The derivation of the dual geodesic equation is simplified using the Helmholtz--Hodge decomposition 
\begin{equation}
\label{unweighted_helmholtz}
\mathfrak{X}(M) = \mathfrak{X}_\mu(M) \oplus \nabla C^\infty(M) \,,
\end{equation}
where $\mathfrak{X}_\mu(M)$ is the space of divergence-free ($\nabla \cdot \bv = 0$) vector fields and the orthogonality is defined with respect to the unweighted $L^2$ inner product
\begin{equation}
\label{unweighted_L2}
\dual{\bu}{\bv} = \int_M g(\bu,\bv )\mu \,.
\end{equation}
In the weighted $L^2$ inner product \eqref{L2_metric}, it follows that if $\bu \in \mathfrak{X}(M)$ satisfies
\begin{equation}
\label{lem:weighted_Helmholtz}
\int_M \rho g(\bu, \bv) \mu = 0 \,,\qquad \forall \bv \in \mathfrak{X}_\mu(M)\,,
\end{equation}
then there exists a $\Sigma \in C^{\infty}(M)$ such that $\bu = \rho^{-1}\nabla \Sigma$. 
\begin{thm}
\label{thm:LC_connection_Hessian}
The geodesic equation defined by the Levi--Civita connection of the metric \eqref{diverence_weighted_metric}
\begin{equation}
\tilde\nabla^{LC}_{\dot\varphi}\dot\varphi = 0 
\end{equation}
is equivalent to the partial differential equation
\begin{equation}
\label{EP_eqs} 
\begin{aligned}
\partial_t \bu + \nabla_{\bu} \bu &= -\frac{1}{\rho}\nabla \Sigma^{LC}\,,
\\
\partial_t \rho + \nabla \cdot(\rho \bu) & = 0 \,,
\\
\rho^{-1} \Sigma^{LC} - \alpha \nabla \cdot \left(\rho^{-1}\nabla\Sigma^{LC}\right) &= \alpha(\mathrm{tr}((\nabla \bu)^2) - \mathrm{Ric}(\bu,\bu)) \,,
\end{aligned}
\end{equation}
in the Eulerian frame where $\dot{\varphi} = \bu \circ \varphi$ and $\rho\mu = \varphi_* \mu$.
\end{thm}
\begin{proof}
The equations \eqref{EP_eqs} can be derived as a form of Euler--Poincar{\'e} (EP) equations, and the proof follows directly from the connection between the EP equations and geodesic equations on Lie groups (see \cite{holm1998euler, holm2009geometric, marsden1997introduction} for further background). By a change of variables, the kinetic energy can be written in the Eulerian frame as
\begin{equation*}
\ell(\bu,\rho) = \frac{1}{2}G^{\psi}_\varphi(\dot\varphi,\dot\varphi)\,,
\end{equation*}
where $\bu = \dot\varphi \circ \varphi^{-1}$ and $\rho \mu = \varphi_* \mu$, variations of the state variable $\varphi$ induce constrained variations of the velocity field and mass density. Let $\eta(t,s)$ be a variation of $\varphi(t)$ such that $\eta(t,0) = \varphi(t)$ and denote $\bv \circ \varphi = \partial_s\eta(t,0) = \delta \varphi$ with fixed endpoints such that $\delta \varphi(0) = \delta \varphi(T) = 0$. It follows then that
\begin{equation*}
\delta \bu = \dot{\bv} + [\bu, \bv] \,, \qquad \delta \varrho = -\mathcal{L}_{\bv} \varrho\,,
\end{equation*}
where $\bv = \delta \varphi \circ \varphi^{-1}$ and $[\cdot, \cdot]$ is the Jacobi-Lie bracket on vector fields. Application of Hamilton's variational principle then gives 
\begin{equation*}
\begin{aligned}
\delta \int_0^T \ell(\bu,\varrho)dt  &= \int_0^T \dual{\frac{\delta \ell }{\delta \bu}}{\dot{\bv} -\text{ad}_{\bu}(\bv)} - \dual{\frac{\delta \ell}{\delta \varrho}}{\mathcal{L}_{\bu}(\varrho)} \, dt 
\\
&= \int_0^T \dual{(\partial_t  + \mathcal{L}_{\bu})\frac{\delta \ell }{\delta \bu}}{\bv}  - \dual{\varrho\otimes d \frac{\delta \ell}{\delta \varrho}}{\bv} \, dt \,,
\end{aligned}
\end{equation*}
which defines the EP equations \eqref{EP_eqs} using the variational derivatives
\begin{equation*}
\begin{aligned}
\frac{\delta \ell}{\delta \bu} &= \rho \bu ^{\flat} - \alpha d (\rho \nabla \cdot \bu) \,,
\\
\varrho \otimes d\frac{\delta \ell}{\delta \varrho} & = \rho d \bigg( \frac{1}{2}|\bu|^2 + \frac{\alpha}{2} (\nabla \cdot \bu)^2\bigg) \otimes \mu  \,,
\end{aligned}
\end{equation*}
and setting the variation of the action to zero. In order to see the particular form of the equations, we write the system as
\begin{equation*}
\underbrace{(\partial_t + \mathcal{L}_{\bu})[\rho \bu^\flat \otimes \mu] - \rho d\frac{1}{2}|\bu|^2 \otimes \mu}_{\mathrm{I}} = \underbrace{(\partial_t + \mathcal{L}_u)[\alpha d(\rho\nabla \cdot \bu)\otimes \mu]}_{\mathrm{II}} + \underbrace{\frac{\alpha}{2}\rho d(\nabla \cdot \bu)^2 \otimes \mu}_{\mathrm{III}}  \,.
\end{equation*}
The $(\mathrm{I})$ terms simplify as
\begin{equation*}
(\mathrm{I}) = [\left( \partial_t\rho + \mathcal{L}_{\bu} \rho + \rho (\nabla \cdot \bu) \right)\bu^{\flat}]\otimes \mu + \rho\big( \partial_t \bu + \mathcal{L}_{\bu} \bu^{\flat} - \frac{1}{2} d|\bu|^2 \big) \otimes \mu = \rho(\partial_t \bu + \nabla_{\bu} \bu)^{\flat} \otimes \mu\,.
\end{equation*}
Using commutativity of the Lie derivative and exterior derivative with the continuity equation for the mass density, we can simplify the $(\mathrm{II})$ term as
\begin{equation*}
\begin{aligned}
\mathrm{(II)} + \mathrm{(III)} &= \alpha d(\partial_t + \mathcal{L}_u)(\rho\nabla \cdot \bu) \otimes \mu + \alpha [d\rho(\nabla \cdot \bu)]\nabla \cdot \bu \otimes \mu  +  \alpha\rho (\nabla \cdot \bu)d(\nabla \cdot \bu) \otimes \mu
\\
& = \cancel{-\alpha d (\rho(\nabla \cdot \bu)^2) \otimes \mu} + \alpha d [\rho (\partial_t + \mathcal{L}_{\bu})(\nabla \cdot \bu)] \otimes \mu + \cancel{\alpha d (\rho (\nabla \cdot \bu)^2) \otimes \mu} \,.
\end{aligned}
\end{equation*}
Combing these two, taking the sharp $\sharp$ operator, and extracting the density, we can write 
\begin{equation}
\label{unsimplified_momentum_eq}
\rho\frac{D\bu}{Dt} - \alpha\nabla \left(\rho \nabla \cdot \left(\frac{D\bu}{Dt}\right)\right) = -\alpha \nabla \left(\rho\left(\nabla \cdot \nabla_{\bu} \bu -\bu \cdot \nabla (\nabla \cdot \bu) \right)\right) \,.
\end{equation}
Letting $\nabla_i u^j = \partial_i u^j + \Gamma^j_{ik}u^k$ be the $i$-th coordinate of the covariant derivative of the velocity field, we can write
\begin{equation*}
\begin{aligned}
\nabla \cdot \nabla_{\bu} \bu -\bu \cdot \nabla (\nabla \cdot \bu) &= \nabla_i(u^j \nabla_j u^i) - u^j \nabla_j \nabla_i u^i
\\
&= (\nabla_i u^j)(\nabla_j u^i) + u^j \nabla_i \nabla_j u^i - u^j \nabla_j \nabla_i u^i
\\
&= \mathrm{tr}((\nabla \bu)^2) + u^j[\nabla_i, \nabla_j] u^i = \mathrm{tr}((\nabla \bu)^2) - \mathrm{Ric}(\bu,\bu)\,,
\end{aligned}
\end{equation*}
where we have introduced the Ricci curvature $\mathrm{Ric}(\cdot,\cdot)$ contracted along the velocity field and used the Ricci identity $[\nabla_i, \nabla_j]u^i = -R_{jk}u^k$ for the commutator of the covariant derivatives. On a flat manifold, this curvature term vanishes. \par
The expression \eqref{unsimplified_momentum_eq} implies that the material derivative satisfies the orthogonality condition \eqref{lem:weighted_Helmholtz}. Defining the elliptic operator
\begin{equation*}
A(\rho)[\bv]  = \rho \bv - \alpha\nabla(\rho (\nabla \cdot \bv))\,,
\end{equation*}
there exists a potential $\Sigma^{LC} \in C^{\infty}(M)$ such that
\begin{equation*}
\mathcal{A}(\rho)\frac{D \bu}{Dt} = \mathcal{A}(\rho)\left[\rho^{-1}\nabla \Sigma^{LC}\right] = \nabla \Sigma^{LC} - \alpha \nabla (\rho \nabla \cdot (\rho^{-1}\nabla \Sigma^{LC})) = -\alpha \nabla(\rho \mathrm{tr}((\nabla \bu)^2) - \rho \mathrm{Ric}(\bu,\bu))\,.
\end{equation*}
Absorbing the minus sign such that $\Sigma^{LC} \mapsto -\Sigma^{LC}$ gives $\frac{D\bu}{Dt} = -\rho^{-1} \nabla \Sigma^{LC}$ with
\begin{equation*}
\rho^{-1} \Sigma^{LC} - \alpha  \nabla \cdot \left( \rho^{-1} \nabla \Sigma^{LC} \right) = \alpha (\mathrm{tr}((\nabla \bu)^2) - \mathrm{Ric}(\bu,\bu)) \,,
\end{equation*}
up to a constant, which establishes the claim. 
\end{proof}

\begin{thm}
\label{thm:entropy_AC}
Given the metric \eqref{diverence_weighted_metric}, the tensor $\tilde{C}^{\sharp}$ defined by the relation \eqref{index_raised_pullback_tensor} satisfies
\begin{equation}
\begin{aligned}
\tilde C^\sharp_{\dot\varphi} \dot\varphi &= \left(\rho^{-1} \nabla \Sigma^{AC}\right) \circ \varphi \,,
\\
\rho^{-1}\Sigma^{AC} - \alpha \nabla\cdot \big(\rho^{-1}\nabla \Sigma^{AC}\big) &= 2\alpha(\nabla \cdot \bu)^2\,.
\end{aligned}
\end{equation}
\end{thm}
\begin{proof}
We want to find $W = \bw \circ \varphi \in T_\varphi \Df$ such that 
\begin{equation*}
G_\varphi^{\psi}(W,V) = \tilde{C}_\varphi(\dot\varphi,\dot\varphi, V) \,, \qquad \forall \, V = \bv \circ \varphi \in T_\varphi\Df  \,.
\end{equation*}
In the ambient space, the Amari--Chentsov cubic tensor defined by the potential \eqref{convex_barrier} is given by
\begin{equation*}
C_{(\varphi,\nu)}((U,\eta), (V,\xi), (W, \zeta)) = \left(0, -2\alpha \int_M \frac{\xi \eta \zeta}{\nu^3} \mu \right)
\end{equation*}
for all $(U,\eta), (V,\xi), (W, \zeta) \in T_{(\varphi,\nu)}\Df \times \Prob$. Combining with the embedding \eqref{embedding}, we get that 
\begin{equation}
\label{a_relation}
\int_M \rho (\bw \cdot \bv + \alpha (\nabla \cdot \bw)(\nabla \cdot \bv)) \mu = -2\alpha \int_M \rho (\nabla \cdot \bu)^2 (\nabla \cdot \bv) \mu \,,
\end{equation}
which implies that $\bw$ satisfies the orthogonality condition \eqref{lem:weighted_Helmholtz}. Then there exists $\Sigma^{AC} \in C^\infty(M)$ such that $\bw = \rho^{-1} \nabla \Sigma^{AC}$ and together with \eqref{a_relation} gives
\begin{equation*}
\int_M \left(\nabla\Sigma^{AC} - \alpha\nabla \big(\rho \nabla \cdot\big(\rho^{-1} \nabla\Sigma^{AC}\big)\big) \right)\cdot \bv \mu  = 2 \alpha \int_M \big(\nabla\rho (\nabla \cdot \bu)^2\big)\cdot \bv \mu  \,,
\end{equation*}
which holds for all $\bv \in \mathfrak{X}(M)$ and where we have used the the tangency condition of $\bv \in \mathfrak{X}(M)$ to drop the boundary integrals after integrating by parts. This then determines the equation for $\Sigma^{AC}$ up to a constant.  
\end{proof}
Combining these two results, we get the following geometric interpretation of the IGR equations and split form of the entropic pressure as arising from the curvature induced by the embedding and the information geometric correction.
\begin{cor} 
\label{cor:pressureless_IGR} 
The restricted dual geodesic equations \eqref{pullback_dual_geodesics} defined by the barrier function \eqref{convex_barrier} are the Lagrangian form of the pressureless IGR equations, written in conservative form 
\begin{equation} \label{presuless_IGR_eqs} 
\begin{aligned} \partial_t \rho \bu + \nabla \cdot(\rho \bu \otimes \bu + \Sigma I) & = 0 \,,
\\
\partial_t \rho + \nabla \cdot (\rho \bu) &= 0 \,,
\\
\rho^{-1} \Sigma - \alpha \nabla \cdot (\rho^{-1} \nabla \Sigma) &= \alpha \left((\nabla \cdot \bu)^2 + \mathrm{tr}(\nabla \bu)^2) - \mathrm{Ric}(\bu,\bu)\right) \,, 
\end{aligned} 
\end{equation} 
with the regularization potential defined by 
\begin{equation} 
\label{combined_entropy} \Sigma = \Sigma^{LC} + \frac{1}{2}\Sigma^{AC} \,.  \end{equation} 
\end{cor} 
\begin{proof}
Using Theorems~\ref{thm:LC_connection_Hessian} and~\ref{thm:entropy_AC}, we can write the pullback dual geodesic equation \eqref{pullback_dual_geodesics} in the Eulerian frame as 
\begin{equation} 
\label{non_conservative_equations} 
\left(\tilde\nabla^{LC}_{\dot{\varphi}} \dot{\varphi} + \frac{1}{2}\tilde{C}^\sharp_{\dot\varphi} \dot\varphi \right) \circ \varphi^{-1} = \partial_t \bu + \nabla_{\bu} \bu + \frac{1}{\rho} \big(\nabla \Sigma^{LC} + \frac{1}{2}\nabla\Sigma^{AC}\big) = 0 \,. 
\end{equation}
Because both $\Sigma^{LC}$ and $\Sigma^{AC}$ satisfy inhomogeneous elliptic equations with the same differential operator, they combine to give $\Sigma$ defined by \eqref{combined_entropy} and the elliptic equation in \eqref{presuless_IGR_eqs}. The pressureless form of the IGR equations \eqref{presuless_IGR_eqs} follows by rewriting the equations in conservative form.
\end{proof}

\section{Thermodynamically Constrained Information Geometric Regularization}
\label{sec:entropy_barrier}

The starting point of the thermodynamically constrained information geometric regularization is to augment the ambient space with an entropy component in the space of volume forms \eqref{volume_forms} and to embed the diffeomorphism group into the extended space using
\begin{equation}
\label{entropy_embedding}
\jmath : \Df \to  \Df \times \Prob \times \mathcal{V}(M) \,, \qquad \jmath(\varphi) = (\varphi, \varphi^*\mu, \varphi_*(\pi_0)) \,,
\end{equation}
where $\pi_0 = \rho_0 s_0 \mu \in \mathcal{V}(M)$ is the initial entropic volume. The embedding \eqref{entropy_embedding} utilizes a mixed relation involving both the pullback of the initial mass density, performing a volumetric regularization on the flow, and the pushforward of the initial entropic volume, which is designed to constrain this regularization to the thermodynamic structure. We consider the relative entropy, described by the Kullback--Leibler divergence \eqref{KL_divergence} from the initial entropy distribution, to define an extended barrier function in the form 
\begin{equation}
\label{entropic_barrier}
\Psi : \Df \times \Prob \times \mathcal{V}(M) \to \mathbb{R} \,,  \qquad \Psi(\varphi,\nu, \pi) = \psi(\varphi, \nu)  + \beta \mathscr{D}_{KL}(\pi \mid\mid \pi_0) \,,
\end{equation}
where $\psi(\varphi,\nu)$ is the barrier function \eqref{convex_barrier} and $\beta >0$ is a parameter controlling the strength of the thermodynamic regularization. Composing \eqref{entropic_barrier} and \eqref{entropy_embedding}, we get that
\begin{equation}
\label{entropic_barrier_functional}
\Psi \circ \jmath(\varphi) =  \frac{1}{2}\int_M |\varphi|^2  \mu + \alpha\int_M \rho\log(\rho)\mu  + \beta \int_M \pi \log\left(\frac{\pi}{\pi_0} \right) \mu  \,,
\end{equation}
where $\rho = \varphi_* \mu$ and $\pi = \pi_0 \circ \varphi^{-1} \cdot J_\mu(\varphi^{-1})$, using the fact that $J_\mu(\varphi)\circ\varphi^{-1} = J_\mu(\varphi^{-1})^{-1}$. This constrains the paths \eqref{pullback_dual_geodesics} to be simultaneously shock mitigating and thermodynamic length minimizing. In Figure~\ref{fig:geometric_visual}, we visualize the discrepancy between these two constraints where the thermodynamic constraint can be seen as distinguishing trajectories in the $(\varphi, J_\mu(\varphi), \pi)$ space.

\begin{figure}[h!]
    \centering
    \includegraphics[height = 6cm, width=\linewidth]{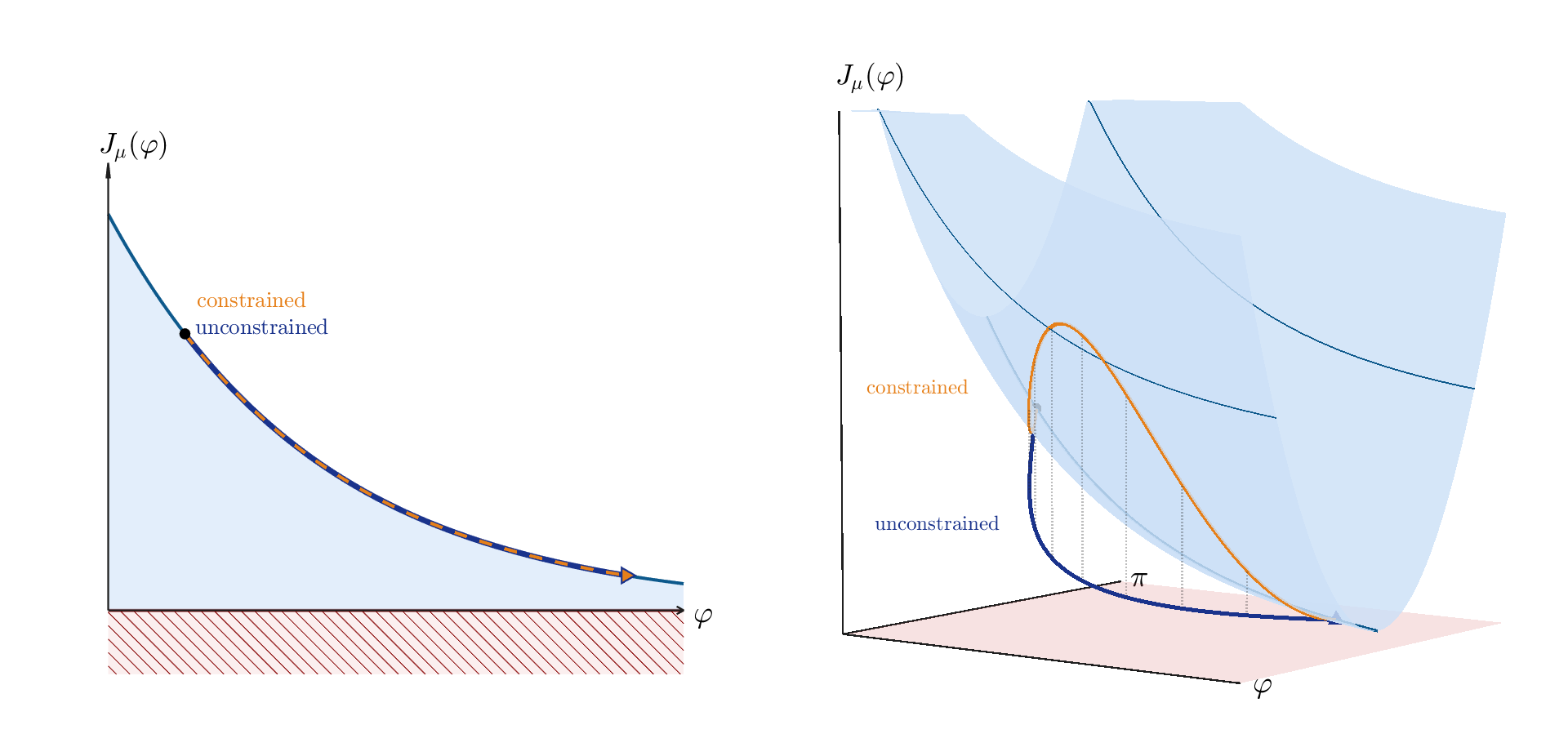}
    \caption{Geometric illustration comparing the thermodynamically constrained and unconstrained paths. Along a slice a constant entropy slice in the $(\varphi, J_\mu(\varphi))$ plane, the two paths collapse to satisfy the same barotropic IGR constraints. In the full state space $(\varphi, J_\mu(\varphi),\pi)$, the additional thermodynamic curvature distinguishes the two trajectories along changing entropy directions.} 
    \label{fig:geometric_visual}
\end{figure}

\subsection{Thermodynamic Constraint}
\label{sec:thermodynamic_constraint}

We first provide a geometric motivation for the particular barrier function \eqref{entropic_barrier_functional} and its thermodynamic constraint on the resulting equations of motion. Entropy minimizing motions can be interpreted in the context of information geometry as selecting admissible paths of minimum thermodynamic length as measured by the Fisher--Rao metric \cite{crooks2007measuring, feng2009far}. This cost also admits a thermodynamic interpretation using a connection between the Wasserstein distance and rates of entropy production \cite{seifert2012stochastic, van2023thermodynamic}. In this section, we discuss these concepts and their relation to the information geometric regularization. We further relate the curvature of the Boltzmann entropy in the Wasserstein geometry to the dynamics of the potential $\chi$ in the TIGRE system \eqref{eq:TIGRE}.\\

\noindent
\emph{The Fisher--Rao Metric}
\vspace{0.1cm}

The entropic embedding \eqref{entropy_embedding} has a natural information-theoretic cost associated to the relative entropy. The KL divergence $\mathscr{D}_{KL}(\varphi_*\pi_0 \mid\mid \pi_0)$ quantifies the information content of the transported distribution $\varphi_* \pi_0$ relative to the prior $\pi_0$, vanishing when the flow leaves the entropy distribution statistically indistinguishable and growing as the flow concentrates or redistributes entropy. The natural geometric structure to measure this statistical indistinguishability comes from the Fisher--Rao metric of information geometry \cite{ay2017information, amari1980theory}, which has also been used to describe a dissipation cost associated with moving between thermodynamic states \cite{crooks2007measuring, feng2009far}. This metric is natural on the space of positive densities because it is the unique invariant metric under the diffeomorphism group action \cite{bauer2016uniqueness}. Because $M$ is oriented, the space of smooth positive densities can be associated to the space of volume forms \eqref{volume_forms} on which all invariant metrics take the form \cite{bauer2016uniqueness}
\begin{equation*}
G_{\nu}(\zeta, \theta) = C_1\int_M \frac{\zeta_\mu}{\nu}\frac{\theta_\mu}{\nu}  \mu + C_2 \int_M \zeta \cdot \int_M \theta
\end{equation*}
for constants $C_1, C_2$ that depend on the total volume $\nu(M)$, where $\zeta = \zeta_\mu \mu$ and $\theta = \theta_\mu \mu$. Restricting to $\Prob$ or setting $C_2 = 0$ yields the Fisher--Rao metric 
\begin{equation}
\label{FR_metric}
G^{fr}_{\nu}(\zeta, \theta) = \int_M\frac{\zeta_\mu}{\nu}\frac{\theta_\mu}{\nu}  \mu  \,.
\end{equation}
The Fisher metric \eqref{FR_metric} can be realized as the Hessian for the family of $\alpha$-divergences, with $\alpha \in (-1,1)$, defined by the convex barrier functionals
\begin{equation}
\label{alpha_divergences}
\mathscr{D}_{\alpha}(\nu \mid\mid \nu_0) = \frac{4}{1- \alpha^2} \left(\int_M \nu_0 \mu - \int_M \nu^{\frac{1-\alpha}{2}} \nu_0^{\frac{1 + \alpha}{2}} \mu \right) \,,
\end{equation}
for fixed $\nu_0 \in \mathcal{V}(M)$. The limits $\alpha \to \pm 1$ recover the Kullback--Leibler divergences $\mathscr{D}_{KL}(\nu\mid\mid\nu_0)$ and $\mathscr{D}_{KL}(\nu_0\mid\mid\nu)$, respectively. The one-parameter family of divergences $\mathscr{D}_\alpha$ possess different parallel transports described by the affine connections $\nabla^{(\alpha)} = \nabla^{LC} - \tfrac{\alpha}{2}C^{\sharp}$ with the Amari--Chentsov tensor
\begin{equation}
C^{fr}(\zeta, \theta, \xi) = - \int_M \frac{\zeta_\mu \theta_\mu \xi_\mu}{\nu^2} \mu  \,.
\end{equation}
The statistical divergences are all equivalent in their definition of a thermodynamic length; however, their impact on the dual geodesic equations varies with the regularization parameter. The IGR equations possess a mixed Fisher--Rao and Wasserstein geometric structure that further connects to a thermodynamic interpretation of the constraint, distinguishing the KL-divergence.  \\

\noindent
\emph{Entropy Production}
\vspace{0.1cm}

Although the adiabatic flow is reversible, the discrepancy between the transported entropy distribution $\varphi_*\pi_0$ and the prior $\pi_0$ still possesses a natural thermodynamic interpretation. In particular, one can associate a thermodynamic cost to a relaxation of $\varphi_* \pi_0 = \nu_2$ back to $\pi_0$, which is a process that would be necessarily irreversible. Transitioning between thermodynamic states $\nu_1, \nu_2 \in \mathcal{V}(M)$ depends on the path taken through state space. Forward and reverse directions of thermodynamic transition are distinguished by the second law of thermodynamics, reflecting a privileged direction (`arrow') of time characterized by the generation of entropy. In classical thermodynamics, a system in energetic exchange with its environment possesses a maximum amount of work extractable as the system relaxes from state $(\rho, \pi)$ to a reference equilibrium $(\rho_0, \pi_0)$ by the \emph{exergy}, which is a state function of the form
\begin{equation}
\mathscr{E}((\rho,\pi) \mid\mid (\rho_0,\pi_0)) = \int_M \rho \left[ e(\rho,s) - e(\rho_0, s_0) + p_0(\rho^{-1} - \rho_0^{-1}) - T_0(s - s_0)\right]\mu \,. 
\end{equation}
The Gouy--Stodola theorem states that the rate of exergy decrease along any irreversible process is proportional to the entropy generation rate. The dissipated work satisfies $W_{diss} = \mathscr{E} - W = T_0 \mathscr{D}_{KL}(\pi\mid\mid \pi_0)$ and measures a gap between this maximum extractable work and the actual work $W$ performed during the transition. The fastest route towards relaxation is through dissipation, and we lastly note how the Wasserstein distance provides a sharp lower bound on the entropy that must be produced. In the case of an overdamped system, the entropy production admits an equivalent path-based interpretation in the context of the dynamical formulation of optimal transport \cite{benamou2000computational} with the Wasserstein distance 
\begin{equation}
W^2_2(\nu_1, \nu_2) = \min_{\bv_t} \int_M \norm{\bv_t}^2\nu_t \,,
\end{equation}
where $\nu_t \in \mathcal{V}(M)$ is a path connecting $\nu_1$ to $\nu_2$ such that $\dot{\nu} = -\mathrm{div}(\nu_t\bv_t)$. In the case of an overdamped system, the irreversible entropy production $S_{gen}$ over a period $\tau$ satisfies bound of the form  $(D \tau) S_{gen} \geq W_2(\nu_1, \nu_2)$ \cite{van2023thermodynamic}, where $D$ is a diffusion coefficient that has even been recently confirmed experimentally \cite{oikawa2025experimentally}. The KL divergence can therefore be similarly interpreted as not necessarily measuring the entropy generated by the flow itself, but rather the information-geometric cost of the equivalent relaxation. It is a proxy for how thermodynamically expensive it would be to undo how the flow map has transformed the entropy distribution. The entropic barrier therefore penalizes flows that create entropic distributions that would be thermodynamically expensive to relax. \\

\noindent
\emph{Curvature in the Wasserstein Geometry}
\vspace{0.1cm}

Considering the entropic volume $\pi(t) \in \mathcal{V}(M)$ to be an element of the space $\mathcal{P}(M)$ where the normalization is now defined by the initial entropic volume $\pi_0(M)$, we can measure the curvature induced by the barrier with respect to the Wasserstein--Otto Riemannian metric structure on space $\Prob$ \cite{otto2001geometry}. In this geometry, tangent vectors at $\pi \in \mathcal{P}(M)$ are associated with the continuity equation source $\zeta = - \nabla \cdot(\pi \nabla \phi)$ for potentials $\phi \in C^{\infty}(M)/\mathbb{R}$, and the metric is defined on the potentials as
\begin{equation*}
\label{Wasserstein_metric}
 G^W_{\pi}(\zeta, \zeta) =  \int_M \pi |\nabla\phi|^2 .
\end{equation*}
The local curvature as expressed by the Hessian of a functional $F: \mathcal{P}(M) \to \mathbb{R}$ in the direction $\zeta = - \nabla \cdot(\pi v)$ is defined by the second variation along Wasserstein geodesics
\begin{equation*}
\mathrm{Hess}_{W_2} F(\pi)[\zeta, \zeta] = \left.\frac{d^2}{dt^2} \right\vert_{t = 0} F(\pi(t)) \,.
\end{equation*}
The KL divergence can be written as
\begin{equation*}
\mathscr{D}_{KL}(\pi \mid \mid \pi_0) = \underbrace{\int_M \pi \log \pi \mu}_{\coloneqq H(\pi)} -  \int_M \pi \log(\pi_0)\mu \,,
\end{equation*}
where the first term is the Boltzmann entropy, and because the second term is linear in $\pi$, its second variation along geodesics vanishes. The KL divergence in the Wasserstein geometry therefore curves in exactly the same way as the Boltzmann entropy of the instantaneous distribution, responding only to the local geometric structure of the current entropy. A similar property does not hold for the $\alpha$-divergences because we are measuring curvature in the Wasserstein geometry. The Hessian of the Boltzmann entropy is given by \cite{lott2009ricci, li2021hessian}
\begin{equation}
\mathrm{Hess}_{W_2} H(\pi)[\zeta,\zeta] = \int_M \pi (|\nabla \bv|^2 + (\mathrm{Ric} + \nabla^2 \log)[\bv,\bv]) \mu \,,
\end{equation}
and the term $(\mathrm{Ric} + \nabla^2 \log\pi)[\bv,\bv]$ is precisely the Bakry--{\'E}mery curvature of the entropy distribution along the velocity field, combining both the geometry of $M$ and the curvature induced by the distribution. This curvature term is also present in the right-hand side of the $\chi$ equation in the TIGRE system \eqref{eq:TIGRE}, indicating that the regularization responds to the instantaneous curvature of the entropy distribution in the direction of the flow, penalizing motions that would concentrate or redistribute entropy in a thermodynamically costly way.

\subsection{TIGRE Equations}

In this section, we provide a derivation of the TIGRE equation as a pullback dual geodesic equation, following an analogous formulation presented for the pressureless IGR equations in Section~\ref{sec:IGR_barotropic}. We first derive the geodesic equations of the metric in the Eulerian frame and then the form of the pullback Amari-Chentsov tensor, beginning with the calculation of the pullback Riemannian metric. 
\begin{lem}
The pullback of the Hessian metric defined by the convex barrier functional \eqref{entropic_barrier_functional} through the embedding \eqref{entropy_embedding} can be expressed in the Eulerian frame as
\begin{equation}
\label{diverence_weighted_metric_entropy}
G^{\Psi}_\varphi(U,V) =  \int_M \rho g(\bu, \bv) \mu + \alpha\int_M \rho (\nabla \cdot \bu)(\nabla \cdot \bv)\mu + \beta \int_M \pi^{-1}(\nabla \cdot( \pi \bu))(\nabla \cdot(\pi \bv)) \mu\,, 
\end{equation}
where $U, V \in T_\varphi \Df$ with $U =  \bu  \circ \varphi$ and $V = \bv \circ \varphi$. 
\end{lem}
\begin{proof}
Let $\gamma(t) \in \Df$ be a curve such that $\gamma(0) = \id$ and $\gamma'(0) =  u \in \mathfrak{X}(M)$. We can then compute variations of the embedding \eqref{entropy_embedding} at $\varphi$ as 
\begin{equation*}
\begin{aligned}
\left.\frac{d}{dt}\right\vert_{t = 0} \jmath(\gamma \circ \varphi) &= \bigg(u \circ \varphi, \left.\frac{d}{dt}\right\vert_{t = 0} \varphi^* \gamma^*(\mu),  \left.\frac{d}{dt}\right\vert_{t = 0} \gamma_*\varphi_* (\pi_0\mu)\bigg) 
\\
&= (u \circ \varphi, \varphi^* \mathcal{L}_u(\mu), \mathcal{L}_u(\pi \mu)) \,,
\end{aligned}
\end{equation*}
from which it follows that 
\begin{equation}
\label{differential_embedding}
D\jmath\vert_{\varphi}[U]  = (u \circ \varphi, (\nabla \cdot u) \circ \varphi \cdot J_\mu(\varphi)\mu, \nabla\cdot (\pi u) \mu)\,.
\end{equation}
The Hessian of the modified convex barrier functional \eqref{entropic_barrier} is now given by 
\begin{equation}
D^2\Psi\vert_{(\varphi, \nu)}\left[(U,\eta, \zeta),(V, \xi, \theta)\right] = D^2\psi\left[(U,\eta), (V, \xi)\right] + \beta \int_M \frac{\zeta_\mu \theta_\mu}{\pi} \mu \,,
\end{equation}
where $\zeta = \zeta_\mu \mu$ and $\theta = \theta_\mu \mu$. Combining these two expression gives the pullback metric 
\begin{equation}
\begin{aligned}
(\jmath^*G^{\Psi})_\varphi(U,V) &=  \tilde{G}^{\psi} _\varphi(U,V) + \int_M \pi^{-1}(\nabla \cdot( \pi u))(\nabla \cdot(\pi v)) \,,
\end{aligned}
\end{equation}
which establishes the claim with $\pi \mu = \varphi_*(\pi_0 \mu)$ and Lemma~\ref{lemma:pullback_metric_barotropic}. 
\end{proof}

\noindent
\emph{Velocity Field Constraint}
\vspace{0.1cm}

\noindent
The relevant orthogonality relation for the thermodynamic extension is with respect to the subspace of divergence-free vector fields defined by
\begin{equation}
\label{entropy_div_free_fields}
\mathfrak{X}_{\mu,\pi}(M) = \left\{\bv \in \mathfrak{X}(M) \,:\, \nabla \cdot \bv = 0 \,, \quad \nabla \cdot (\pi \bv) = 0\,, \quad \bv \cdot n \vert_{\partial M } = 0 \right\} \,.
\end{equation}
These two simultaneous constraints distinguish a thermodynamic decomposition from the Helmholtz decomposition used in the IGR formulation and the orthogonal complement of \eqref{entropy_div_free_fields} admits an explicit characterization with an additional constraint on the curl of the velocity field.
\begin{lem}
\label{lem:thermo_Helmholtz}
Let $(M,g)$ be a compact Riemannian manifold potentially with boundary $\partial M$ and suppose that $\rho,\pi \in C^{\infty}(M)$ are positive functions. Let $\bw \in \mathfrak{X}(M)$ be such that 
\begin{equation}
\int_M \rho g(\bw, \bv) \mu = 0 \qquad \forall \bv \in \mathfrak{X}_{\mu,\pi}(M)
\end{equation}
and suppose that there exists a potential $\tilde\chi \in C^{\infty}(M)$ such that $d(\rho \bw^{\flat}) =  d\pi \wedge d \tilde\chi$. Then there exists $\Sigma, \chi \in C^\infty(M)$ such that
\begin{equation}
\label{thermo_helmholtz_decomposition}
w = \frac{1}{\rho}\left(\nabla \Sigma + \pi \nabla \chi\right) \,.
\end{equation}
\end{lem}
\begin{proof}
We can first write $\rho \bw \coloneqq \bsym{z} = \nabla \Sigma' + \bu$ where $\bu \in \mathfrak{X}_\mu(M)$ using the Helmholtz--Hodge decomposition, and the problem reduces to finding an appropriate decomposition for the divergence-free component. The one-form $\theta = \bu^{\flat} - \pi d \tilde\chi$ since 
\begin{equation*}
d\theta = d\bu^{\flat} - d\pi \wedge d \tilde\chi = d\bsym{z}^\flat - d \pi \wedge d \tilde\chi = 0\,,
\end{equation*}
and therefore $\theta = d\phi + \zeta_H$ by the Hodge decomposition where $\zeta_H \in \Omega^1(M)$ is a harmonic form. Using the orthogonality condition $\bsym{z} \in \mathfrak{X}_{\mu, \pi}(M)^{\perp}$ we then get that 
\begin{equation*}
\int_M g(\zeta_H, \bv)\mu  = 0  \qquad \forall \bv \in \mathfrak{X}_{\mu,\pi}(M) \,.
\end{equation*}
Since $\zeta_H \in \mathfrak{X}_\mu(M)$ this implies the tangential component of $\zeta^{\sharp}_H$ to the level set of $\pi$ vanishes and therefore $\zeta_H = f d\pi$. Altogether this gives 
\begin{equation*}
\rho \bw = \nabla (\Sigma' + \phi) + \pi \nabla \tilde\chi + f \nabla \pi = \nabla( \Sigma' + \phi + f \pi) + \pi \nabla (\tilde\chi - f) \,,
\end{equation*}
and labelling $\chi = \tilde{\chi} - f$ and $\Sigma = \Sigma' + \phi + f\pi$ establishes the claim. 
\end{proof}

\noindent
\emph{Gauge Freedom}
\vspace{0.1cm}

\noindent
The decomposition \eqref{thermo_helmholtz_decomposition} has an infinite-dimensional kernel, defined by the gauge transformation
\begin{equation*}
(\Sigma, \chi) \mapsto (\Sigma + H(\pi), \chi + F(\pi)) \,,
\end{equation*}
where $F$ is a arbitrary smooth function on the range of $\pi$ and $H' = -\pi F'$, which is invariant under the decomposition. Under this transformation, the expression $\nabla \Sigma + \pi\nabla \chi$ determining the velocity field is invariant, and therefore has no effect on the momentum regularization. In our numerical scheme no gauge fixing is applied, however, selecting an appropriate gauge condition could be used to simply the resulting linear system by selecting an equivalence class representative that is best conditioned for the solver for instance. We regard this as an interesting direction for future work. 

\begin{thm}
\label{thm:geodesics_TIGRE}
The geodesic equation defined by the Levi--Civita connection of the metric \eqref{diverence_weighted_metric_entropy} on $\Df$ is equivalent to the system of partial differential equations
\begin{equation}
\label{presureless_IGR_entropic}
\begin{aligned}
\partial_t \bu + \nabla_{\bu} \bu &= \rho^{-1}(\nabla \Sigma + \pi \nabla \chi) \,,
\\
\partial_t \rho + \nabla \cdot (\rho \bu) &= 0 \,,
\\
\partial_t \pi + \nabla \cdot (\pi \bu) &= 0 \,,
\\
\rho^{-1}\Sigma - \alpha\, \nabla \cdot \left(\rho^{-1}(\nabla \Sigma + \pi \nabla \chi)\right) &= \alpha( \mathrm{tr}((\nabla \bu)^2) - \mathrm{Ric}[\bu,\bu]) 
\\
\chi - \beta \mathrm{div}_{\pi}\left(\rho^{-1}(\nabla \Sigma + \pi \nabla \chi)\right) &=  \beta \big((\nabla^2\log \pi + \mathrm{Ric})[\bu,\bu] - \mathrm{tr}((\nabla \bu)^2) + \frac{1}{2}\mathrm{div}^2_\pi(\bu)\big)\,.
\end{aligned}
\end{equation}
\end{thm}
\begin{proof}
As in the proof of Theorem~\ref{thm:LC_connection_Hessian}, we derive these equations from the Euler--Poincar{\'e} equations of the reduced Lagrangian
\begin{equation}
\ell(\bu,\rho,\pi) = \frac{1}{2}\left.\left(\jmath^* D^2\Psi\right)\right\vert_{(\varphi, \varphi^*\mu, \varphi_* \pi_0)}\left[\dot\varphi, \dot\varphi\right] \,,
\end{equation}
which now contains an additional transported variable. The EP equations are given by
\begin{equation*}
(\partial_t + \mathcal{L}_{\bu}) \left(\frac{\delta \ell}{\delta \bu} \otimes \mu \right) = \rho d \frac{\delta \ell}{\delta \rho} \otimes \mu + \pi d\frac{\delta \ell}{\delta \pi}\otimes \mu \,.
\end{equation*}
Denoting $\sigma = \pi^{-1}\nabla \cdot( \pi \bu)$, we can write the variational derivatives as
\begin{equation*}
\frac{\delta \ell }{\delta \bu } = \rho \bu^{\flat} - \alpha d (\rho \nabla \cdot \bu) - \beta \pi d\sigma \,, \quad \frac{\delta \ell}{\delta \rho} = \frac{1}{2}|\bu|^2 + \frac{1}{2}\alpha(\nabla \cdot \bu)^2 \,, \quad \frac{\delta \ell}{\delta \pi} = -\beta \left(\bu \cdot \nabla \sigma + \frac{1}{2}\sigma ^2 \right)\,.
\end{equation*}
Rearranging again in terms of the $L^2$ kinetic and regularization terms, we can write
\begin{equation*}
\mathrm{(I)} = \underbrace{(\partial_t + \mathcal{L}_{\bu}) (\alpha d(\rho \nabla \cdot \bu) \otimes \mu)}_{\mathrm{(II)}}  +  \underbrace{\alpha \rho d \frac{1}{2}(\nabla \cdot \bu)^2}_{\mathrm{(III)}} +  \underbrace{(\partial_t + \mathcal{L}_{\bu}) (\beta \pi d\sigma\otimes \mu)}_{\mathrm{(IV)}} + \underbrace{\pi d \frac{\delta \ell}{\delta \pi} \otimes \mu}_{\mathrm{(V)}} \,,
\end{equation*}
where $(\mathrm I) - (\mathrm{II}) - (\mathrm{ III})$ combine as in the proof of Theorem~\ref{thm:LC_connection_Hessian}. The term $(\mathrm{IV})$ simplifies using the continuity equation for the entropy \eqref{pi_continuity} to give
\begin{equation*}
\begin{aligned}
(\mathrm{IV}) + (V) &= \beta (\partial_t + \mathcal{L}_{\bu})  d \sigma \otimes  \pi\mu + \pi d \frac{\delta \ell}{\delta \pi} \otimes \mu 
\\
&= \beta \pi d(\partial_t + \mathcal{L}_{\bu})   \sigma \otimes \mu - \beta\pi d \big(\bu \cdot \nabla \sigma + \frac{1}{2}\sigma ^2 \big) \otimes\mu
\\
&=  \beta \pi d\big(\partial_t\sigma + \bu \cdot \nabla \sigma - \bu \cdot \nabla \sigma - \frac{1}{2}\sigma^2\big) \otimes \mu
= \beta \pi d\big(\partial_t\sigma - \frac{1}{2}\sigma^2\big) \otimes \mu \,.
\end{aligned}
\end{equation*}
Using the continuity equation again we can simplify $\partial_t\sigma$ as
\begin{equation*}
\begin{aligned}
\partial_t \sigma &= -\pi^{-2} \partial_t \pi (\nabla \cdot (\pi \bu)) + \pi^{-1} \nabla \cdot (\partial_t \pi \bu) + \pi^{-1} \nabla \cdot (\pi \partial_t \bu)
\\
& = \sigma^2 - \pi^{-1} \nabla \cdot (\bu \nabla \cdot(\pi \bu)) + \pi^{-1} \nabla \cdot (\pi \partial_t \bu)
\\
& = \sigma^2 - \pi^{-1}\nabla \cdot (\pi \sigma \bu)  + \pi^{-1} \nabla \cdot (\pi \partial_t \bu)
\\
&=  -\bu \cdot \nabla \sigma + \pi^{-1} \nabla \cdot(\pi \partial_t \bu)\,.
\end{aligned}
\end{equation*}
Defining the operator
\begin{equation}
\mathcal{A}(\rho, \pi)[\bv] = \rho \bv - \alpha \nabla (\rho \nabla \cdot \bv) - \beta\pi \nabla(\pi^{-1} \nabla \cdot (\pi \bv))\,,
\end{equation}
we can combine these results to give
\begin{equation}
\label{the_formula}
\mathcal{A}(\rho,\pi) \left[\frac{D\bu}{Dt}\right] = -\alpha\nabla (\rho Q(\bu)) - \beta \pi\nabla \big(Q_\pi(\bu) + \frac{1}{2}\sigma^2\big)\,,
\end{equation}
where we have defined entropy weighted term
\begin{equation*}
Q_\pi(\bu) = \bu \cdot \nabla \sigma  - \pi^{-1} \nabla \cdot (\pi \nabla_{\bu} \bu) \,.
\end{equation*}
We can simplify this operation by noting that $\sigma$ is the divergence of $\bu$ with respect to the weighted volume form $\pi \mu$, meaning $\sigma = \mathrm{div}_\pi(\bu)$ where 
\begin{equation*}
\mathcal{L}_{\bu} (\pi \mu) = \mathrm{div}_\pi(\bu) \pi\mu \,.
\end{equation*}
It can be shown then that the tensor
\begin{equation*}
F_\pi(\bu) = \nabla \bu + \bu \otimes \nabla \log \pi
\end{equation*}
satisfies the relation $\mathrm{tr}(F_\pi(\bu)) = \mathrm{div}_\pi(\bu)$, allowing us to write
\begin{equation}
Q_\pi(\bu) = \bu \cdot \nabla \mathrm{div}_\pi(\bu) - \mathrm{div}_\pi (\nabla_{\bu} \bu) = \mathrm{tr}\left(\nabla_{\bu} F_\pi(\bu) - F_\pi(\nabla_{\bu} \bu) \right) \,,
\end{equation}
using the commutativity between the trace and the covariant derivative. We can then simplify
\begin{equation*}
\nabla_{\bu} F_\pi(\bu) - F_\pi(\nabla_{\bu} \bu) =  \nabla_{\bu} \nabla \bu+ \cancel{\nabla_{\bu} \bu \otimes \nabla \log \pi} + \bu \otimes \nabla_{\bu}(\nabla \log \pi) -  \nabla (\nabla_{\bu} \bu) - \cancel{\nabla_{\bu} \bu \otimes \nabla \log \pi} \,.
\end{equation*}
We then note that the definition of the Hessian $\nabla^2 f$ of the scalar function $\log \pi$ gives us directly that
\begin{equation*}
\mathrm{tr}\left(\bu \otimes \nabla_{\bu}\left(\nabla \log\pi\right)\right)  = \nabla^2 \log\pi (\bu,\bu) \,.
\end{equation*}
On a Riemannian manifold, the commutator of the covariant derivatives can be simplified usign the coordinate expression as
\begin{equation*}
\begin{aligned}
(\nabla_{\bu}\nabla \bu)^i_{\,j} - (\nabla(\nabla_{\bu} \bu))^i_{\,j}&= u^k \nabla_k \nabla_j u^i - u^k \nabla_j \nabla_k u^i - (\nabla_j u^k)(\nabla_k u^i)
\\
& = u^k(\nabla_k \nabla_j - \nabla_j \nabla_k) u^i - [(\nabla u)^2]^i_{\,j} = u^k R^i_{\,jkl} u^l - [(\nabla u)^2]^i_{\,j} \,.
\end{aligned}
\end{equation*}
Taking the trace then gives yields
\begin{equation}
Q_\pi(\bu) = \nabla^2 \log\pi(\bu,\bu) + \mathrm{Ric}(\bu,\bu) - \mathrm{tr}((\nabla \bu)^2) \,.
\end{equation}
Combining with \eqref{the_formula} we get that
\begin{equation}
\label{the_expression2}
\begin{aligned}
\mathcal{A}(\rho,\pi)\left[\frac{D\bu}{Dt}\right] &= \rho\frac{D\bu}{Dt} - \alpha \nabla \left(\rho \nabla \cdot \frac{D\bu}{Dt}\right) - \beta\pi \nabla\left(\pi^{-1} \nabla \cdot \left(\pi \frac{D\bu}{Dt}\right)\right)\,
\\
&= -\alpha \nabla (\rho(\mathrm{tr}((\nabla \bu)^2) - \mathrm{Ric}(\bu,\bu)))
\\
&- \beta \pi \nabla \big(\nabla^2 \log\pi(\bu,\bu) + \mathrm{Ric}(\bu,\bu) - \mathrm{tr}((\nabla \bu)^2)+ \frac{1}{2}\mathrm{div}_\pi(\bu)^2\big)\,,
\end{aligned}
\end{equation}
and we can see that $\rho \frac{D\bu}{Dt}$ satisfies the conditions for Lemma~\ref{lem:thermo_Helmholtz}. Then writing 
\begin{equation*}
\rho \frac{D\bu}{Dt} = -\nabla \Sigma-\pi \nabla \chi
\end{equation*}
and substituting into the expression \eqref{the_expression2} gives the coupled elliptic system of equations in \eqref{presureless_IGR_entropic} and establishing the claim. 
\end{proof}

\begin{thm}
\label{thm:entropy_AC_thermo}
Given the metric \eqref{diverence_weighted_metric_entropy} and barrier function \eqref{entropic_barrier}, the tensor $\tilde{C}^{\sharp}$ defined by the relation \eqref{index_raised_pullback_tensor} satisfies
\begin{equation}
\label{AC_entropy_weighted}
\begin{aligned}
\tilde C^\sharp_{\dot\varphi} \dot\varphi &= \left(\rho^{-1} (\nabla \Sigma^{AC} + \pi \nabla \chi^{AC} )\right)\circ \varphi \,,
\\
\rho^{-1}\Sigma^{AC} - \alpha \nabla \cdot \left(\rho^{-1}(\nabla \Sigma^{AC} + \pi \nabla \chi^{AC})\right) &= 2\alpha (\nabla \cdot \bu)^2 \,,
\\
\chi^{AC} - \beta \pi^{-1} \nabla \cdot\left(\pi \rho^{-1}(\nabla \Sigma^{AC} + \pi \nabla \chi^{AC})\right) &=  \beta \left(\mathrm{div}_\pi(\bu)\right)^2 \,.
\end{aligned}
\end{equation}
\end{thm}
\begin{proof}
Again we seek $W = \bw \circ \varphi \in T_\varphi\Df$ such that
\begin{equation}
\label{the_relation}
G^\Psi_\varphi(W,V) = \tilde C_\varphi(\dot\varphi, \dot\varphi,V) \,,
\end{equation}
for all $V = \bv \circ \varphi \in T_\varphi \Df$. In the ambient space, the Amari-Chentsov tensor is given by 
\begin{equation*}
C^\Psi_{(\varphi, \nu, \pi)}\left((U, \eta, \zeta), (V, \xi, \sigma), (W, \theta, \kappa)\right)= \left(0 , -2\alpha \int_M \frac{\eta \xi \theta}{\nu^3} \mu, -\beta \int_M \frac{\zeta \sigma \kappa}{\pi^2} \mu \right) \,.
\end{equation*}
Composing with the differential of the embedding \eqref{differential_embedding} and a change of variables gives
\begin{equation*}
C_\varphi(\dot\varphi,\dot\varphi, V) = -2\alpha \int_M \rho (\nabla \cdot \bu)^2 (\nabla \cdot \bv)\mu  - \beta \int_M (\mathrm{div}_\pi(\bu))^2 (\nabla \cdot (\pi \bv)) \mu
\end{equation*}
A direct computation based on the relation \eqref{the_relation} then gives
\begin{equation}
\begin{aligned}
\label{relation2}
\rho \bw  - \alpha \nabla (\rho (\nabla \cdot \bw)) - \beta \pi \nabla (\pi^{-1} \nabla \cdot (\pi \bw)) = 2\alpha \nabla (\rho (\nabla \cdot \bu)^2) + \beta \pi \nabla (\mathrm{div}_\pi(u))^2 \,.
\end{aligned}
\end{equation}
It follows that $\rho \bw$ satisfies the conditions for Lemma~\ref{lem:thermo_Helmholtz} and we can write $\rho \bw = \nabla \Sigma^{AC} + \pi \nabla \chi^{AC}$. Substituting into the relation \eqref{relation2} yields the system of equations \eqref{AC_entropy_weighted}, establishing the claim. 
\end{proof}
We can now combine these results to form the TIGRE equations, explicitly including the thermodynamic potential terms. 
\begin{cor}
The forced pullback dual geodesic equation defined by the barrier function \eqref{entropic_barrier_functional} as 
\begin{equation}
\label{dual_geodesic_TIGRE}
\tilde\nabla^{LC}_{\dot\varphi} \dot\varphi + \frac{1}{2} \tilde{C}^{\sharp}_{\dot\varphi}\dot\varphi = -\left(\nabla h(\rho,s) - T(\rho,s)\nabla s\right) \circ \varphi 
\end{equation}
is the Lagrangian form of the TIGRE system \eqref{eq:TIGRE}.
\end{cor}
\begin{proof}
Expressing the dual geodesic equation \eqref{dual_geodesic_TIGRE} in the Eulerian frame and using the pressureless form \eqref{presureless_IGR_entropic} and the expression for the $L^2$ gradient of the pullback potential energy \eqref{full_pullback_potential}, we get that
\begin{equation}
\frac{D\bu}{Dt} = - \rho^{-1}(\nabla h(\rho,s) - T(\rho,s) \nabla s) - \rho^{-1}(\nabla \Sigma + \pi \nabla \chi) \,,
\end{equation}
where we have identified
\begin{equation*}
\Sigma = \Sigma^{LC} + \frac{1}{2}\Sigma^{AC} \,, \qquad \chi = \chi^{LC} + \frac{1}{2} \chi^{AC}
\end{equation*}
are the combined potentials. The right-hand sides of the equations for the potentials follow from a direct computation combining the solutions determined by Theorems~\ref{thm:geodesics_TIGRE} and~\ref{thm:entropy_AC_thermo}, and the conservative form follows similarly as the IGR formulation. 
\end{proof}

\subsection{Properties of the Regularized Equations}
\label{sec:properties_regularized_equations}
The potentials $(\Sigma,\chi)$ defining the regularization play structurally distinct roles in the dynamics. While the volumetric regularization potential $\Sigma$ compensates for the compressive motions, the potential $\chi$ is sensitive to the thermodynamic structure of the flow, forcing only the non-isentropic motions. Here we elaborate on the conservative, thermodynamic, and dynamic consequences of these properties of the regularization in the TIGRE system \eqref{eq:TIGRE} in comparison with the IGR equations \eqref{full_compressible_euler_regularized}.\\

\noindent
\emph{Conservative Properties}
\vspace{0.1cm}

Along with the total mass of the system, the total momentum, energy, and entropy are conserved quantities of the full compressible Euler equations
\begin{equation}
\mathcal{P}(t) = \frac{d}{dt}\int_M \rho(t) \bu(t) \mu \,, \quad \mathcal{E}(t) = \int_M \rho(t)\left(\frac{1}{2}|\bu(t)|^2 + e(t)\right)\mu \,, \quad \mathcal{S}(t) = \int_M \pi(t) \,. 
\end{equation}

\begin{prop}
The IGR equations \eqref{full_compressible_euler_regularized} conserve total mass, momentum, and energy \eqref{total_energy_density}. The total entropy satisfies
\begin{equation}
\frac{d}{dt}\int_M \pi \mu = \int_M \bu \cdot \nabla \left(\vartheta^{-1}\Sigma\right) \mu\,.
\end{equation}

\end{prop}
\begin{proof}
The conservative properties follow as a consequence of the divergence theorem, which can be read directly from the conservative flux form the equations \eqref{full_compressible_euler_regularized}. The evolution of the entropy was noted in \cite{barham2025hamiltonian}, which is pointwise defined by 
\begin{equation}
\label{entropy_source_IGR}
\partial_t \pi + \nabla \cdot (\pi \bu) = - (\vartheta^{-1}\Sigma) \nabla \cdot \bu \,,
\end{equation}
using the thermodynamic relations \eqref{thermodynamic_quantities}.
\end{proof}

\begin{prop}
The TIGRE equations \eqref{eq:TIGRE} conserve total mass and entropy. The total momentum satisfies
\begin{equation}
\frac{d}{dt}\int_M \rho \bu dx =  -\int_M \pi \nabla \chi dx\,,
\end{equation}
and the total energy satisfies
\begin{equation}
\frac{d}{dt}\int_M \rho\bigg(\frac{1}{2}|\bu|^2 + e(\rho,s)\bigg) dx = -\int_M \bu \cdot ( \nabla\Sigma + \pi \nabla \chi) dx \,.
\end{equation}
\end{prop}
\begin{proof}
The conservative properties follow directly from the equations, and the evolution of the total momentum follows directly from the momentum equation. The total energy density evolution proof follows directly from the equations and thermodynamic relations. 
\end{proof}
Compressions of the flow becomes sources or sinks of entropy in the IGR equations \eqref{full_compressible_euler_regularized}. The conservation of energy exhibited by the equations \eqref{full_compressible_euler_regularized} can be seen as a constraint equation used to close the system rather than a structural property of the regularization itself. The TIGRE system similarly does not possess a conservation of the total energy density; however, it conserves the total entropy by construction.  Neither of the systems necessarily dissipates the total energy or entropy, however, and in our numerical experiments, we observe both increasing and decreasing of the total energy in the TIGRE system \eqref{eq:TIGRE} and the total entropy for the IGR system \eqref{full_compressible_euler_regularized}. This behaviour can be attributed to a sign-indefinite character of the regularization \cite{cao2023information, barham2025hamiltonian}; serving as a source of regularization energy in compressive regions or a sink in expansive regions. In both systems, the natural Hamiltonian defined by the kinetic and potential energies is not conserved and instead satisfies the evolution
\begin{equation}
\label{total_energy_evolution}
\begin{aligned}
\frac{d}{dt} \frac{1}{2}G^{\psi}_\varphi(\dot{\varphi},\dot\varphi) + U(\varphi) &= G^{\psi}_\varphi(\nabla^{LC}_{\dot\varphi}\dot\varphi, \dot\varphi) + G^{\psi}_\varphi(\nabla^{\psi} U(\varphi), \dot\varphi)
\\
& = G^{\psi}_\varphi(\nabla^*_{\dot{\varphi}}\dot\varphi + \nabla^{\psi}  U(\varphi), \dot\varphi)  - \frac{1}{2}G^{\psi}_\varphi(C^\sharp_{\dot\varphi}\dot\varphi,\dot\varphi) 
\\
& =  G^{\psi}_\varphi(\nabla^{\psi}U(\varphi) - \nabla^{G}U(\varphi), \dot\varphi)  - \frac{1}{2}C_\varphi(\dot\varphi,\dot\varphi,\dot\varphi)  \,,
\end{aligned}
\end{equation}
where we used the notation $\nabla^{\psi}$ to define the gradient with respect to the pullback metric and $\nabla^GU(\varphi)$ as the gradient with respect to the $L^2$ metric as in the Newton's equations \eqref{Newtons_eqs} and the fact that
\begin{equation}
\label{TIGRE_equations_Lagrangian_form}
\tilde\nabla^*_{\dot\varphi} \dot\varphi = -\nabla^G U(\varphi)
\end{equation}
in the last inequality for either system \eqref{full_compressible_euler_regularized} or \eqref{eq:TIGRE}. The cubic term arising due to the Amari--Chentsov tensor in \eqref{total_energy_evolution} serves as a source of regularization energy in compressive regions or a sink in expansive regions and contrasts the sign-definite of dissipative viscous regularizations. This sign-indefinite character of the information geometric regularization is fundamental to its formulation as an inviscid regularization and is a consequence of the use of the dual geodesic evolution.\\

\noindent
\emph{Acoustic Wave Propagation}
\vspace{0.1cm}

The use of the $L^2$ gradient in the forcing of the pullback dual geodesic equations \eqref{TIGRE_equations_Lagrangian_form} has an important structural consequence in preserving the linear acoustic wave propagation in the flow. The use of the gradient of the Hessian metric, which results in dynamics of the form
\begin{equation*}
\label{pullback_dual_geodesics_barotropic}
\tilde\nabla^*_{\dot\varphi} \dot\varphi = -\nabla^{\psi} U(\varphi)\,,
\end{equation*}
can be shown to introduce a non-local pressure into the momentum equation, which compromises the acoustic wave dispersion relations of the full compressible Euler equations. A similar property is noted in \cite{barham2025hamiltonian}, where it was also shown that the IGR equations conserve the acoustic wave propagation of the Euler equations. Using a similar line of reasoning, we can show that the TIGRE equations \eqref{eq:TIGRE} also preserve the acoustic wave propagation of the full compressible Euler equations. In particular, we have the following proposition. 

\begin{prop}
The linearized characteristics of the TIGRE equations \eqref{eq:TIGRE} about a constant background state support acoustic waves with dispersion relation $\tau = u_0 \cdot \xi \pm c_0|\xi|$, where $c_0$ is the isentropic speed of sound and $\xi$ is the wave vector. The regularization potentials $(\Sigma, \chi)$ do not contribute to the acoustic spectrum at linear order. 
\end{prop}
\begin{proof}
We consider a constant background state $U_0 = (\rho_0, u_0, \pi_0)$ and a perturbation of the solution $U = U_0 + \epsilon \tilde{U}$ with a plane wave ansatz
\begin{equation*}
\tilde{U}(x,t) = \hat{U} e^{i\xi \cdot x - \tau t} \,, 
\end{equation*}
where $\hat{U} = (\hat\rho, \hat u,\hat\pi)$ are expressed in Fourier space. We first note that with the background state we get a steady state solution with $(\Sigma, \chi) = (0,0)$ because the right-hand sides of the equations are zero since $u_0$ and $\pi_0$ are constant and the zero solution is valid for both elliptic equations. Linearizing the elliptic subsystem about $(\Sigma, \chi) = \epsilon(\tilde{\Sigma}, \tilde\chi)$, we get the that the left-hand sides satisfy
\begin{equation*}
\begin{aligned}
\rho^{-1} \Sigma - \alpha \nabla \cdot(\rho^{-1}(\nabla \Sigma + \pi \nabla \chi)) &= \epsilon \rho_0^{-1}\tilde\Sigma - \alpha \epsilon\rho_0^{-1} \nabla \cdot\left( \nabla \tilde\Sigma + \pi_0 \nabla \tilde\chi\right) + \mathcal{O}(\epsilon^2) \,,
\\
\chi - \beta \pi^{-1}\nabla \cdot \left(\pi\rho^{-1}(\nabla \Sigma + \pi \nabla \chi)\right) &= \epsilon \tilde\chi - \epsilon \beta \rho^{-1}_0 \nabla \cdot \left( \nabla \tilde\Sigma + \pi_0 \nabla \tilde\chi \right) + \mathcal{O}(\epsilon^2)  \,.
\end{aligned}
\end{equation*}
Since the background velocity is spatially constant, the terms $\mathrm{div}(u)^2$ and $\mathrm{tr}((\nabla u)^2)$ are $\mathcal{O}(\epsilon^2)$. In the $\chi$ equation, we see that 
\begin{equation*}
\mathrm{div}_\pi(u_0 + \epsilon \tilde{u}) = \mathrm{div}(u_0 + \epsilon \tilde u) + \epsilon(u_0 + \epsilon\tilde{u}) \cdot \nabla \tilde\pi = \epsilon \mathrm{div}(\tilde u) + \epsilon u_0 \pi_0^{-1} \nabla \tilde\pi + \mathcal{O}(\epsilon^2) \,,
\end{equation*}
and hence its square is $\mathcal{O}(\epsilon^2)$. In the other term, we have that
\begin{equation*}
\nabla^{2}\log \pi [u_0 + \epsilon \tilde{u}, u_0 + \epsilon \tilde{u}] = \nabla^{2}\log \pi_0 [u_0, u_0] + 2 \epsilon \nabla^{2}\log \pi_0 [u_0, \tilde u ] + \mathcal{O}(\epsilon^2)
\end{equation*}
and the first terms vanish because $\nabla^2 \log(\pi_0) = 0$. Substituting into the system of equations \eqref{eq:TIGRE}, replacing $\partial_t \to -i \tau$ and $\nabla \to i \xi$, and introducing a Doppler-shifted frequency $\hat \tau = \tau - u_0 \cdot \xi$, the linearized equations become an algebraic system of the form 
\begin{equation}
\label{linearized_system}
-\hat\tau \rho_0 \hat u + c_0^2 \xi \hat \rho + \frac{\partial p}{\partial s} \xi \hat s + \xi \hat \Sigma + \pi_0 \xi \hat\chi = 0 \,, \quad 
- \hat \tau \hat \rho + \rho_0 \xi \cdot \hat u = 0 \,, \quad 
-\hat \tau \hat \pi + \pi_0 \xi \cdot \hat u = 0 \,,
\end{equation}
and the elliptic subsystem at $\mathcal{O}(\epsilon)$ becomes 
\begin{equation*}
\begin{pmatrix}
\rho_0^{-1}(1 + \alpha |\xi|^2 & \alpha \pi_0 \rho_0^{-1} |\xi|^2
\\
\beta \pi_0^{-1} \rho_0^{-1} |\xi|^2 & 1 + \beta \pi^{-1}_0 \rho_0^{-1} |\xi|^2 
\end{pmatrix}
\begin{pmatrix}
\hat \Sigma 
\\
\hat \chi
\end{pmatrix}
= 
\begin{pmatrix}
0
\\
0
\end{pmatrix}
\,.
\end{equation*}
The determinant of this matrix is given by
\begin{equation*}
\det \mathcal{L}(\xi) = \rho_0^{-1}\left( 1 + (\alpha + \beta \pi_0^{-1} \rho_0^{-1}) |\xi|^2\right) > 0\,,
\end{equation*}
and therefore $\mathcal{L}(\xi)$ is invertible for all $\xi$ allowing us to conclude the $\hat{\Sigma} =  \hat\chi = 0$, simplying the momentum equation relation. Note then that the continuity equation relations give us the relation $\hat \rho/\rho_0 = \hat\pi/\pi_0$ and since $\pi = \rho s$, linearizing we get $\hat\pi = s_0 \hat\rho + \rho_0 \hat s$ and therefore
\begin{equation*}
\frac{\hat\pi}{\pi_0} = \frac{\hat\rho}{\rho_0} + \frac{\rho_0 \hat s}{\pi_0} = \frac{\hat\rho}{\rho_0} + \frac{\hat s}{s_0} \implies \hat s = 0 \,.
\end{equation*}
The system \eqref{linearized_system} then reduces to
\begin{equation*}
\begin{aligned}
-\hat\tau \rho_0 \hat u + c_0^2 \xi \hat \rho  &= 0 \,,
\\
- \hat \tau \hat \rho + \rho_0 \xi \cdot \hat u & = 0 \,,
\end{aligned}
\end{equation*}
which implies that $\hat{u}$ is proportional to the wave vector $\xi$ and there is no equation forcing the transverse directions $\hat{u}_{\perp}$ to be vorticity modes. Letting $\hat{u}_{\parallel}$ be the longitudinal direction, the only nontrivial solution of the equations is given by $\hat{\tau} = \pm c_0 |\xi|$, and restoring the Doppler shift, we get $\tau = u_0 \cdot \xi \pm c_0 |\xi|$, establishing the claim.  
\end{proof}

% Moreover, it is sufficient to treat the case $u_0 = 0$ due to a Galilean invariance around constant backgrounds. To see this, consider a Galilean boost of the form $(x,t,u) \mapsto (x - u_0 t, t, u - u_0)$ where $u_0 \in \mathbb{R}^2$ is a constant vector field. The continuity equations then transforms as
% \begin{equation*}
% \partial_t \rho + \nabla \cdot (\rho u) = \partial_{t'}\rho + \nabla' \cdot (\rho u') + \underbrace{\nabla' \cdot (\rho u_0) - u_0 \nabla \rho'}_{= \rho \nabla' \cdot u_0 = 0} = \partial_{t'} \rho + \nabla' \cdot (\rho u') \,,
% \end{equation*}
% and similarly for the entropy equation. Similarly for the momentum equation we see that
% \begin{equation*}
% \begin{aligned}
% \partial_t(\rho u) + \nabla \cdot (\rho u \otimes u) &= (\partial_{t'} - u_0 \cdot \nabla')(\rho u' + \rho u_0) = \partial_{t'}(\rho u') - u_0 \cdot \nabla '(\rho' u') + u_0 \partial_{t'} \rho - u_0(u_0 \cdot \nabla') \rho
% \\
% & + \nabla' \cdot (\rho u' \otimes u') + u_0 \nabla' \cdot (\rho u_0) + (u_0 \cdot \nabla')(\rho u')  + (u_0 \cdot \nabla' \rho) u_0
% \\
% & = \partial_t(\rho u') + \nabla \cdot (\rho u' \otimes u') + u_0( \partial_{t'}\rho + \nabla' \cdot(\rho u') + u_0(u_0 \cdot \nabla' \rho - u_0 \cdot \nabla' \rho)
% \\
% & = \partial_t(\rho u') + \nabla \cdot (\rho u' \otimes u') = -\nabla'\cdot((p + \Sigma)I) - \pi \nabla' \chi \,.
% \end{aligned}
% \end{equation*}

\noindent
\emph{Potential Vorticity}
\vspace{0.1cm}

It is desirable that the regularization to mitigate shock formation still maintain the vortex-wave interactions described by the transport of the potential vorticity. In the non-barotropic setting in a three-dimensional domain, the Ertel potential vorticity $q = \rho^{-1} \omega \cdot \nabla s$, where $\bsym{\omega} = \nabla \times u$ is the vorticity vector field, is a material invariant of the full compressible Euler equations. This can be readily observed using a coordinate-free interpretation with differential forms, where we can write
\begin{equation}
\label{potential_vorticity}
 q = \rho^{-1} \star (\omega \wedge ds)\,, 
\end{equation}
with $\omega \wedge ds$ being the wedge product between the vorticity two-form $\omega = du^{\flat}$ and the exterior derivative of the specific entropy and $\star: \Omega^3(M) \to \Omega^0(M)$ is the Hodge star operation. Taking the exterior derivative of the evolution equation for the velocity one-form $u^{\flat}$ in the compressible Euler equations gives us
\begin{equation*}\
\partial_t u^{\flat} + \mathcal{L}_u u^{\flat} = -d\big(h - \frac{1}{2}|u|^2\big) + \vartheta ds \implies \partial_t \omega + \mathcal{L}_u \omega = d\vartheta \wedge ds\,,
\end{equation*}
where we have used the commutativity of the Lie derivative and exterior derivative. Using the fact that $\rho q = \star \omega \wedge ds \implies \rho q\mu = \omega \wedge ds$ and the continuity equation for $\rho$, we have
\begin{equation}
\label{material_invariance_PV}
\begin{aligned}
 \rho(\partial_tq + \mathcal{L}_u q)\mu &= q(\partial_t + \mathcal{L}_u)\rho \mu + \rho(\partial_tq + \mathcal{L}_u q)\mu 
\\
&=(\partial_t + \mathcal{L}_u) \rho q \mu 
\\
& = (\partial_t + \mathcal{L}_u)\omega \wedge ds  + \omega \wedge d (\partial_t + \mathcal{L}_u)s
\\
& = d \vartheta \wedge ds \wedge ds = 0 \,,
\end{aligned}
\end{equation}
which implies the advection of the potential vorticity \eqref{potential_vorticity}. 
\begin{prop}
The potential vorticity \eqref{potential_vorticity} in the TIGRE system \eqref{eq:TIGRE} is a material invariant.
\end{prop}
\begin{proof}
Writing the momentum equation in differential form and taking the exterior derivative, we see
\begin{equation*}
\partial_t \omega + \mathcal{L}_u \omega = d\vartheta \wedge ds - d(\rho^{-1}\pi d\chi) = (d \vartheta + d\chi) \wedge ds \,,
\end{equation*}
using the anti-symmetry of the wedge product. The conservation of the potential vorticity then follows using the conservation of entropy and a similar manipulation as equation \eqref{material_invariance_PV}.
\end{proof}

\begin{prop}
The potential vorticity \eqref{potential_vorticity} in thermodynamic IGR equations \eqref{full_compressible_euler_regularized} satisfies
\begin{equation*}
\partial_t q + u \cdot \nabla q = - \frac{1}{\rho} (\nabla \times u) \cdot \nabla \left(\frac{\Sigma}{\rho \vartheta}(\nabla \cdot u)\right)\,.
\end{equation*}
\end{prop}
\begin{proof}
This property follow as a consequence of the loss of conservation of the entropy exhibited by the system \eqref{full_compressible_euler_regularized}. In particular, because $\pi$ satisfies \eqref{entropy_source_IGR}, we see that
\begin{equation*}
(\partial_t + \mathcal{L}_u)\omega \wedge ds = d\vartheta \wedge ds \wedge ds + \omega \wedge d(\partial_t + \mathcal{L}_u) s = -\omega \wedge d\left(\frac{\Sigma}{\vartheta \rho} (\nabla \cdot u)\right) \,.
\end{equation*}
The proof then follows by using $\rho q\mu = \omega \wedge ds$ and taking the Hodge star on both sides to give
\begin{equation*}
\star \rho(\partial_t q + \mathcal{L}_u q) \mu = \rho(\partial_t q + \mathcal{L}_u q) = - \star \omega \wedge d\left(\frac{\Sigma}{\vartheta \rho} (\nabla \cdot u)\right) = -(\nabla \times u) \cdot \nabla \left(\frac{\Sigma}{\rho \vartheta}(\nabla \cdot u)\right)\,,
\end{equation*}
which establishes the formula. 
\end{proof}
The dynamical invariance of the potential vorticity in the TIGRE equations is a corollary of the conservation of entropy and can be seen as an additional benefit to conserving the entropy when formulating the information geometric regularization. 

\section{Numerical Experiments}
\label{sec:numerical_experiments}
In this section, we describe a suite of numerical experiments that investigate and validate the properties of the TIGRE equations \eqref{eq:TIGRE} in comparison to the thermodynamic IGR equation \eqref{full_compressible_euler_regularized}. We use the same numerical method as \cite{cao2023information}, a Richtmyer Lax--Wendroff time-stepping scheme to compute the solution. The first experiment(\ref{sec:sod_experiments}) is a one-dimensional, smooth, and periodicized Sod shock tube initial condition as used in \cite{barham2025hamiltonian} (section 1.4) to demonstrate a spurious cusp-like singularity formation in the density and internal energy of the thermo-IGR equations after the collision of two wave fronts for the equations \eqref{full_compressible_euler_regularized}. The TIGRE equations are demonstrated to mitigate this cusp-singularity by constraining the dynamics to be thermodynamic length minimizing. The second experiment (\ref{sec:front_collisions}) illustrates the dynamical similarity of the two equations with a two-dimensional experiment of colliding pressure waves, demonstrating numerically how acoustic wave dispersion is preserved in both sets of equations. The third experiment (\ref{sec:KH_instability}) illustrates the interactions between pressure waves and vortex-dominated motion in a Kelvin--Helmholtz instability. We observe that the TIGRE regularization approach introduces less dispersive error at vortex filaments while still maintaining a similar dynamics in the velocity and pressure fields. These experiments all provide a strong indication of the efficacy of including a thermodynamic constraint into the information geometric regularization.\\

\noindent
\emph{Numerical Methods.} Our numerical experiments are performed using a Richtmyer Lax--Wendroff (LW) iteration of the form
\begin{equation}
\begin{aligned}
\bsym{q}^{n+1/2}_{i + 1/2} &= \frac{1}{2}(\bsym{q}^n_i + \bsym{q}^n_{i+1}) - \frac{\Delta t}{2\Delta x} \left(\bsym{F}(\bsym{q}^n_{i+1}) - \bsym{F}(\bsym{q}^n_i)\right) + \frac{\Delta t}{2}\left(\bsym{S}(\bsym{q}^n_i) + S(\bsym{q}^n_{i+1}) \right) \,,
\\
\bsym{q}^{n+1}_i &= \bsym{q}^n_i - \frac{\Delta t}{\Delta x}\left(\bsym{F}(\bsym{q}^{n+1/2}_{i + 1/2}) - \bsym{F}(\bsym{q}^{n+1/2}_{i -1/2})\right) + \Delta t \bsym{S}\left(\frac{1}{2}(\bsym{q}^{n+1/2}_{i+1} + \bsym{q}^{n+1/2}_{i-1})\right) \,,
\end{aligned}
\end{equation}
where we use the notation
\begin{equation*}
\bsym{q} = (\rho, \rho u, \pi)^{\top}, \quad \bsym{F}(q) = (\rho u, \rho u^2 + p(\rho,s) + \Sigma)^\top, \quad \bsym{S}(q) = (0, -\pi \nabla \chi, 0)^\top \,,
\end{equation*}
for the state vector, flux, and source term respectively. The domain is taken to be $[0,1)^2$ with periodic boundary conditions, and the solution is computed on a uniform mesh of size $\Delta x= 1/N_x$. The derivative in the source term is approximated using a centred finite difference, and the elliptic system for $(\Sigma^n, \chi^n)$ needed to evaluate the fluxes is constructed as follows. Letting $D_x u^n_{ij}$ and $D_y u^n_{ij}$ denote the centred finite-difference approximations of the velocity field in the coordinate directions, we can write the approximation of the right-hand side of the equations for $\Sigma$ using the expressions
\begin{equation*}
\begin{aligned}
[\mathrm{div}(u)^2]_{ij} &= (D_x u^x_{ij} + D_y u^y_{ij})^2 \,,
\\
[\mathrm{tr}((\nabla u)^2)]_{ij} &= (D_x u^x_{ij})^2 + 2 (D_x u^y_{ij})(D_y u^x_{ij}) + (D_y u^y_{ij})^2 \,.
\end{aligned}
\end{equation*}
The entropy-weighted divergence term is computed using a cell-averaging of the form 
\begin{equation*}
[\mathrm{div}_\pi(u)]_{i,j} = \frac{1}{\pi_{i,j}}\left(\frac{\bar\pi_{i+1/2,j} u^x_{i+1,j} - \bar\pi_{i-1/2,j} u^x_{i-1,j}}{\Delta x} + \frac{\bar\pi_{i,j+1/2} u^y_{i,j+1} - \bar\pi_{i,j-1/2} u^y_{i,j-1}}{\Delta y}\right) \,,
\end{equation*}
with the notation $\bar\pi_{i\pm1/2,j} = (\pi_{i,j} + \pi_{i \pm 1,j})/2$ and similarly for $\bar\pi_{i,j \pm 1/2}$. The Hessian of $\log \pi$ evaluating along the velocity field is approximated as
\begin{equation*}
\begin{aligned}
[\nabla^2 \log \pi[u,u]]_{i,j} &= (u_{i,j}^x)^2 \left(\log \pi_{i+1,j} - 2\log \pi_{ij} + \log \pi_{i-1,j}\right)/\Delta x^{2}
\\
&+ u^x_{i,j} u_{i,j}^y \left(\log \pi_{i+1, j+1} - \log \pi_{i-1,j+1} - \log \pi_{i+1,j-1} + \log \pi_{i-1,j-1}\right)/(2 \Delta x\Delta y)
\\
&+ (u_{i,j}^y)^2 \left(\log \pi_{i,j+1} - 2\log \pi_{ij} + \log \pi_{i,j-1}\right)/\Delta y^{2} \,.
\end{aligned}
\end{equation*}
We discretize the weighted divergence operators defining the system for $(\Sigma, \chi)$ using cell-centred finite-differences with face-averaged coefficients as
\begin{equation}
\label{weighted_divergence_discrete}
\begin{aligned}
\left[\nabla \cdot (g \nabla f)\right]_{ij} &= \frac{1}{2\Delta x^{2}}\left(\bar g_{i+1/2,j}(f_{i+1,j} - f_{i,j}) + \bar g_{i-1/2,j}(f_{i,j} - f_{i-1,j})  \right)
\\
&+ \frac{1}{2\Delta y^2}\left(\bar g_{i,j+1/2}(f_{i,j+1} - f_{ij}) + \bar g_{i,j-1/2}(f_{i,j} - f_{i,j-1})  \right)\,.
\end{aligned}
\end{equation}
The elliptic system for the potentials at each iteration can be expressed as the linear system
\begin{equation}
\begin{pmatrix}
D_\rho + \alpha L_{11} & \alpha L_{12}
\\
\beta L_{21} D_\pi & I + \beta L_{22}D_\pi
\end{pmatrix}
\begin{pmatrix}
    \Sigma^n 
    \\
    \chi^n
\end{pmatrix}
= \begin{pmatrix}
\alpha (\mathrm{div}(u^n)^2 + \mathrm{tr}((\nabla u^n)^2)) 
\\
\beta(\mathrm{div}^2_\pi(u^n) - \mathrm{tr}((\nabla u^n)^2) + \nabla^2 \log(\pi^n)[u^n,u^n]
\end{pmatrix}\,,
\end{equation}
where $D_\rho = \mathrm{diag}(\rho^{-1}_{ij})$ and $D_\pi = \mathrm{diag}(\pi^{-1}_{ij})$ are diagonal matrices of pointwise multiplications and we have used the notation $L_{kl}$ to denote the $N \times N$ symmetric matrix representing the weighted divergence operators $L_{kl} = - \nabla \cdot (g_{kl}\nabla)$ as defined by \eqref{weighted_divergence_discrete} with $g_{11} = \rho^{-1}$, $g_{12} = g_{21} = \pi \rho^{-1}$, and $g_{22} = \pi^2 \rho^{-1}$. Writing the matrix defining the linear system as $A$ and the right-hand side as $(f_1, f_2)$, we can split the system as $A = A_D + A_N$ where $A_D$ collects all centre node contributions and $A_N$ collects all the off-diagonal neighbour couplings the cells $(i,j)$. The splitting induces the following fixed-point iteration
\begin{equation}
\label{fixed_point_iteration}
A_D \begin{pmatrix}
    \tilde\Sigma^{k+1}
    \\
    \tilde\chi^{k+1}
\end{pmatrix}
= - A_N \begin{pmatrix}
    \tilde\Sigma^{k+1}
    \\
    \tilde\chi^{k+1}
\end{pmatrix} + \begin{pmatrix}
    f_1
    \\
    f_2
\end{pmatrix}\,,
\end{equation}
where $f_1, f_2$ corresponds to the right-hand side of the system and $(\tilde\Sigma^0, \tilde \chi^0) = (\Sigma^n, \chi^n)$ is the initialization. Because $A_D$ is block diagonal with $2\times 2$ blocks, we can compute each solution exactly via Cramer's rule on the blocks. We use a similar Gauss--Seidel fixed-point iteration similar to those in prior IGR work \cite{cao2023information, wilfong2025simulating} to solve the system. The method uses a red-black colouring of the mesh to improve the convergence rate. On a small $4\times 4$ grid this colouring looks like
\begin{center}
    \begin{tikzpicture}
  % Grid lines
  \foreach \x in {0,1,2,3,4} {
    \draw[gray!50] (\x, 0) -- (\x, 4);
  }
  \foreach \y in {0,1,2,3,4} {
    \draw[gray!50] (0, \y) -- (4, \y);
  }

  \foreach \i in {0,1,2,3} {
    \foreach \j in {0,1,2,3} {
      \pgfmathsetmacro{\s}{int(mod(\i+\j,2))}
      \pgfmathsetmacro{\cx}{\i+0.5}
      \pgfmathsetmacro{\cy}{\j+0.5}
      \ifnum\s=0
        % Red cell
        \fill[red!25]   (\i,\j) rectangle (\i+1,\j+1);
        \node[font=\small] at (\cx,\cy) {R};
      \else
        % Black cell
        \fill[black!25] (\i,\j) rectangle (\i+1,\j+1);
        \node[font=\small, white] at (\cx,\cy) {B};
      \fi
    }
  }
  \def\cx{1.5} \def\cy{1.5}

  \foreach \nx/\ny in {2.5/1.5, 0.5/1.5, 1.5/2.5, 1.5/0.5} {
    \draw[->, thick, blue!70] (\cx,\cy) -- (\nx,\ny);
  }

  \draw[red!80, very thick] (1,1) rectangle (2,2);

  % Axis labels
  \foreach \i in {0,1,2,3} {
    \node[below, font=\footnotesize] at (\i+0.5, 0) {$\i$};
    \node[left,  font=\footnotesize] at (0, \i+0.5) {$\i$};
  }
  \node[below] at (2, -0.4) {$i$};
  \node[left]  at (-0.4, 2) {$j$};
\end{tikzpicture}
\end{center}
At each iteration, the values of either the red or black cells are frozen, and the fixed-point iteration is performed consecutively between coloured cells. The regularization parameters scale as $\alpha \sim \beta \sim \Delta x^2$, yielding a well-conditioned system. The number of fixed-point iterations requiring convergence was approximately halved by using an extrapolation $(\tilde{\Sigma}^0, \tilde\chi^0) = (2\Sigma^n-\Sigma^{n-1}, 2\chi^n - \chi^{n-1})$ as the initial guess at each time step. In practice, we observed only a few fixed-point iterations were needed for convergence. 

The time stepping is performed adaptively using a CFL condition of the form
\begin{equation*}
\Delta t = C\frac{\Delta x}{\norm{|u| + \sqrt{\partial p/\partial \rho}}_{\infty}}
\end{equation*}
with a constant $C$ chosen based on the experimental setup. We will refer to the solution of the thermodynamic extension of the IGR equations \eqref{full_compressible_euler_regularized} using the LW scheme as LW-IGR and the application to our new TIGRE equations as LW-TIGRE.\\

\noindent
\emph{Metrics of Comparison.} As quantitative comparison of the two regularization approaches, we can verify the conservative properties of the numerical solutions with the following approximations of the total entropy, the total energy, and the momentum of the solution 
\begin{equation*}
\mathcal{S}^n = \sum_{i,j}  \pi_{i,j}^n \Delta x \Delta y\,, \quad \mathcal{E}^n = \sum_{i,j} \rho_{i,j}^n\bigg(\frac{1}{2}|u_{i,j}^n|^2 + \frac{p^{n}_{i,j}}{1- \gamma}\bigg) \Delta x \Delta y \,, \quad 
\mathcal{P}^n = \sum_{i,j} \rho_{i,j}^n u^n_{i,j} \Delta x \Delta y \,.
\end{equation*}
As a measure of the oscillatory behaviour in the solution, we consider the discrete total variation of the density measured as
\begin{equation*}
\text{TV}_h(\rho^n) = \sum_{x_i} \sqrt{(D_x \rho^n)^2 + (D_y \rho^n)^2} \Delta x \Delta y \,,
\end{equation*}
where $D_x \rho^n$  and $D_y \rho^n$ are centred finite-difference approximations in the coordinate directions. We also consider diagnostics for the acoustic wave propagation and the turbulent energy cascade behaviour based on spectra of the discrete Fourier transformed pressure and velocity fields. In particular, we computed radially averaged power spectra, defined by the square modulus of the coefficients over the shells $\Omega_{b} = \{\bsym{k} = (k_x,k_y): b \leq |\bsym{k}| \leq b+1\}$ as
\begin{equation*}
P(k_b,t_n) = \frac{1}{|\Omega_b|} \sum_{\bsym{k} \in \Omega_b} |\hat{p}^n(\bsym{k})|^2 \,, \quad E(k_b, t_n) = \frac{1}{2|\Omega_b|} \sum_{\bsym{k} \in \Omega_b} |\hat{\bsym{u}}^n(\bsym{k})|^2 \,.
\end{equation*}
These spectral diagnostics are used to assess the effects of the regularization on the acoustic wave propagation and turbulent energy cascades respectively.

\subsection{Smoothed Sod Shock Tube}
\label{sec:sod_experiments}
In this experiment, we consider a smoothed form of the Sod shock tube test as described in \cite{barham2025hamiltonian}. The TIGRE system \eqref{eq:TIGRE} in one dimension reduces to
\begin{equation}
\begin{aligned}
\partial_t (\rho u ) + \partial_x \left( \rho u^2 + p(\rho,s) + \Sigma)\right) & = - \pi \partial_x \chi\,,
\\
\partial_t \rho + \partial_x (\rho u) & = 0\,,
\\
\partial_t \pi + \partial_x (\pi u) &= 0 \,,
\\
\rho^{-1}\Sigma - \alpha \partial_x \left(\rho^{-1}(\partial_x\Sigma + \pi \partial_x \chi)\right) &= 2 \alpha (\partial_x u)^2\,,
\\
\chi - \beta \pi^{-1} \partial_x \left(\pi \rho^{-1}(\partial_x \Sigma + \pi \partial_x \chi)\right) &=  \beta\left((\pi^{-1}\partial_x (\pi u))^2 - (\partial_x u)^2 + u^2 \partial_{xx} \log \pi \right) \,.
\end{aligned}
\end{equation}
As initial conditions, we use a smoothed, periodicized version of the Sod tube in the form 
\begin{equation*}
\begin{aligned}
\psi(x) &= \frac{1}{2} \left(\tanh((x - x_L)/\epsilon) - \tanh((x - x_R)/\epsilon) \right)\,,
\\
\rho_0(x) &= \rho_L - (\rho_L - \rho_R)\psi(x) \,, \quad p_0(x)  = p_L - (p_L - p_R) \psi(x)\,, \quad u_0(x) = 0 \,.
\end{aligned}
\end{equation*}
The parameters of the experiment are taken to be $(x_L, x_R) = (0.25,0.75)$, $(\rho_L, \rho_R) = (1, 0.125)$, $(p_L, p_R) = (1, 0.1)$, and $\epsilon = 0.03$. The initial total energy density is determined by the relation \eqref{total_energy_ideal_gas} for the LW-IGR simulation, and the ideal gas law \eqref{ideal_gas_law} is used to determine the initial entropy. We used the constants $\kappa = 1$, $\gamma = 1.4$, and $c_v = 2.5$ for the LW-TIGRE simulation. As a baseline comparison, we consider the Lax--Friedrichs (LF) scheme defined by
\begin{equation*}
\bsym{q}^{n+1}_i = \frac{1}{2}(\bsym{q}^n_{i+1} + \bsym{q}^n_{i-1}) - \frac{\Delta t}{2\Delta x}\left(\bsym{F}(q^n_{i+1}) - \bsym{F}(q^n_{i-1})\right) \,,
\end{equation*}
applied to the compressible Euler equations. The LF scheme is prone to a diffusive error, whereas the LW scheme is prone to a dispersive error \cite{leveque1992numerical}. It was shown that the IGR regularization dampens the magnitude of the dispersive oscillations near sharp moving fronts \cite{cao2023information}, which we also observed in the LW-TIGRE scheme. The solutions for LW-IGR and LW-TIGRE are computed using $N_x = 500$ grid points and a final integration time of $t = 0.5$. The solution at three different snapshot in time is shown in Figure \ref{fig:sod_shock_solution}, where we compare the density, velocity, and pressure, which are all shared observable quantities from each method. The wave front speeds and pressure of the LW-IGR and LW-TIGRE equation have a slight discrepancy, whereas the LW-TIGRE solution does not develop a cusp singularity in the density compared to the LW-IGR solution. We illustrate how the cusp singularity remains a property of the solution, rather than a numerical artifact by considering a refinement up to a resolution of $N_x = 1024$ grid points in Figure~\ref{fig:sod_shock_solution}. In Figure~\ref{fig:diagnostics_sod}, we provide the metrics of comparison, where we observe machine precision accuracy of the momentum for both equations, conservation of total entropy for the LW-TIGRE, and conservation of total energy for the LW-IGR equation, which is consistent with the discussion of Section~\ref{sec:properties_regularized_equations}. This test was also performed by Barham et al.~\cite{barham2025hamiltonian}, where the cusp-singularity was not resolved, indicating that the incorporation of thermodynamics into the information geometric regularization offers a systematic resolution to this dynamical behaviour. 

\begin{figure}
    \centering
    \includegraphics[width=\linewidth]{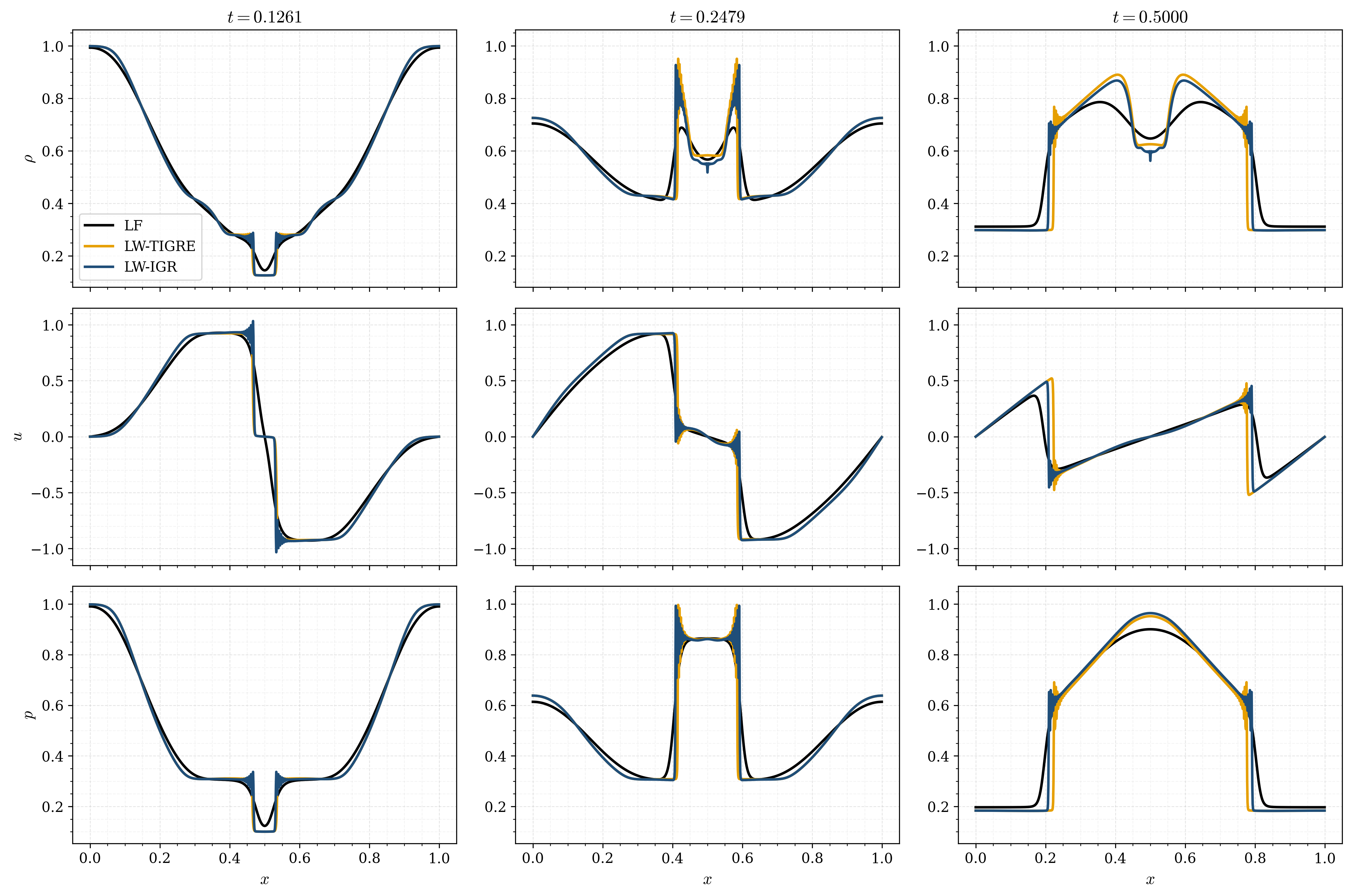}
    \includegraphics[width = 8cm, height = 5cm]{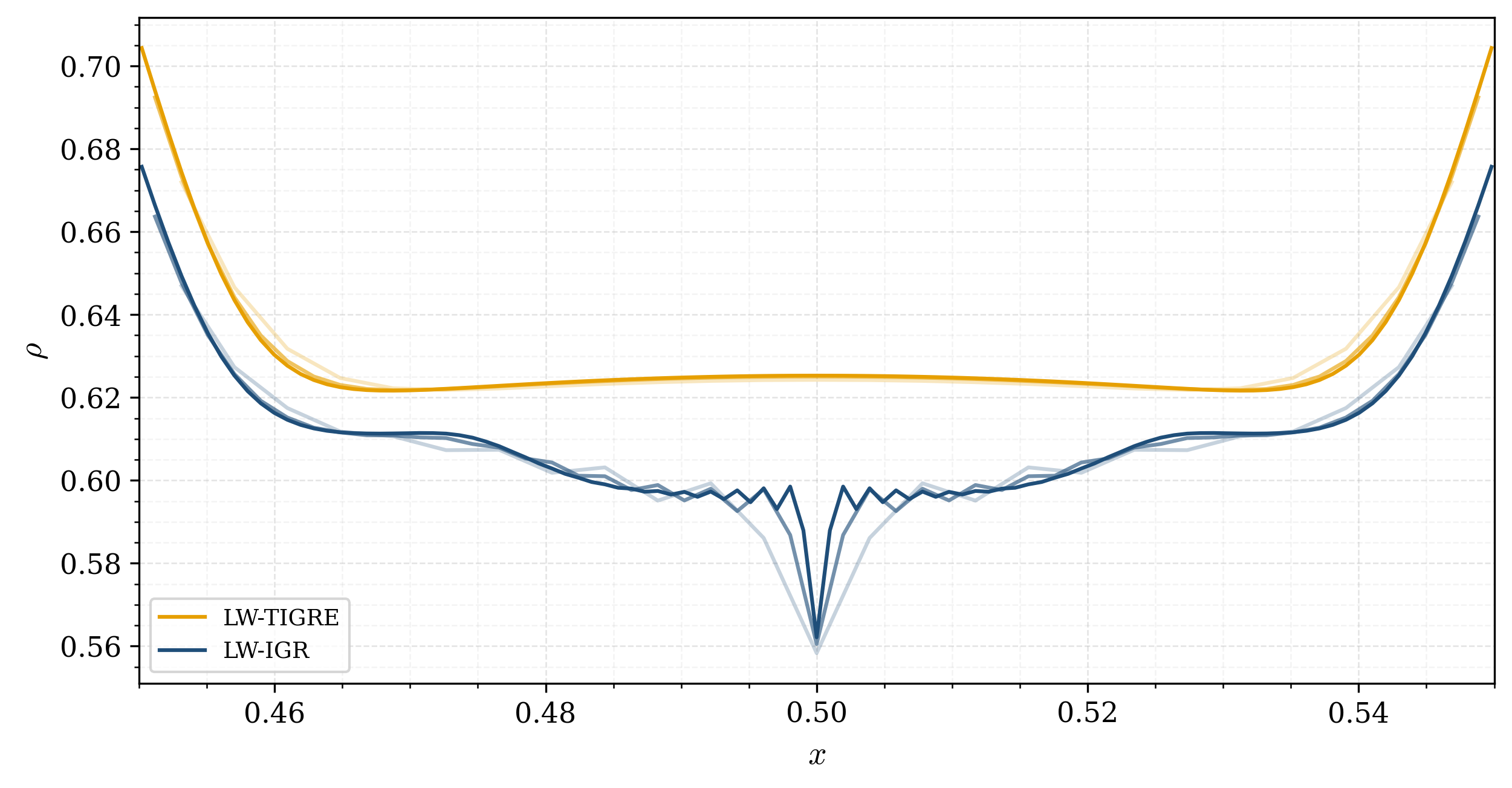}
    \caption{\emph{Top}: Smoothed Sod shock tube experiment. The solution $(\rho, u, p)$ at three different snapshots in time. The wave front speeds and pressure of the LW-IGR and LW-TIGRE equation have a slight discrepancy, whereas the LW-TIGRE solution does not develop a cusp singularity in the density compared to the LW-IGR solution. \emph{Bottom}: Zoom of the mass density at the final integration time, where we illustrate the permanence of the cusp-singularity under grid refinement.}
    \label{fig:sod_shock_solution}
\end{figure}

\begin{figure}
    \centering
    \includegraphics[width = \linewidth]{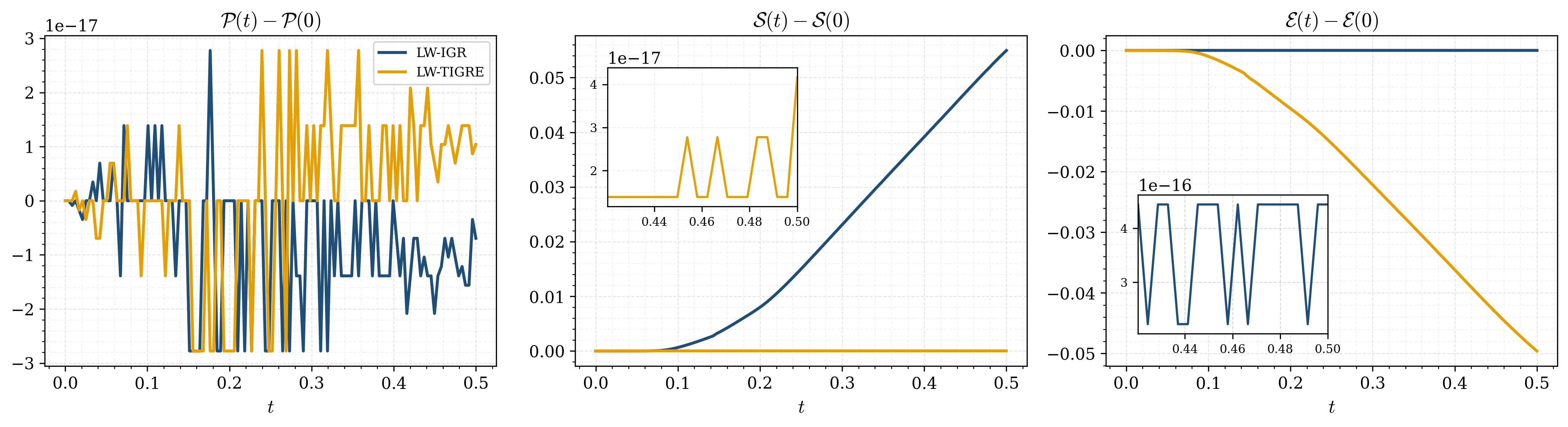}
    \caption{Conservation of the momentum (left), entropy (middle), and total energy (right) for the smoothed Sod shock tube experiment.}
    \label{fig:diagnostics_sod}
\end{figure}

\subsection{Acoustic Wave Propagation}
\label{sec:front_collisions}
In this numerical experiment, we consider the evolution of two localized pressure perturbations in a medium at rest in a two-dimensional periodic domain to investigate the behaviour of a complex acoustic wave propagation. We consider the background density to be constant and the initial velocity field is set to zero. As perturbation, we use two compactly supported profiles, one with radial symmetry and one with a variable profile, defined respectively as
\begin{equation*}
\begin{aligned}
b_1(x,y) &= \exp(-(1 - r_1^2)^{-1})\mathbb{I}_{r_1 < 1} \,, \qquad r_1^2 =  \frac{1}{\epsilon^2}\left((x- 0.3)^2 + (y - 0.4)^2\right)\,,
\\
b_2(x,y) &= \exp(-(1- q_2^2)^{-1})\mathbb{I}_{q_2 < 1} \,, \qquad q_2^2 = \frac{1}{\varepsilon(\theta)^2}\left((x - 0.7)^2 + (y - 0.6)^2\right)\,,
\end{aligned}
\end{equation*}
where $\epsilon = 0.01$ and $\varepsilon(\theta) = 0.1(1 + 0.18 \cos(4\theta))$ with $\theta = \arctan 2(y - 0.6, x - 0.7)$.  The asymmetry between these two fronts injects an anisotropic structure into the initial condition and angular modes into the spectrum.  The initial density and entropy are then defined by $\rho_0 = 1 + 0.1(b_1 + b_2)$ and $\pi_0 = 0.2 + 0.1(b_1 + b_2)$. The pressure perturbations associated with these initial conditions generate acoustic waves that transport the entropy density, simulating an asymmetric collisions between wavefronts. We consider $\Delta x = \Delta y = 1/512$ and compute the solution to $t = 1$, which we show in Figure~\ref{fig:collision_front}. The solutions are nearly identical, indicating how the acoustic wave propagation for each regularization is consistent. We can further validate this property quantitatively by assessing the pressure power spectrum diagnostics of the two solutions in Figure~\ref{fig:AW_spectral_diagnostics}. The regularization exhibited a similar dynamical behaviour with a small discrepancy at the highest wave numbers. As a coarse-grained comparison, we compute also the total spectral mass defined by the total power spectrum and the variance weighted mean wave number
\begin{equation*}
k_{rms}(t) = \sqrt{\frac{\sum_b k^2 P(k)}{\sum_b P(k)}} \,,
\end{equation*}
as further indication of the consistency between the acoustic wave propagation dynamics between each regularization. The total spectral mass is observed to oscillate over the simulation, while the $k_{rms}(t)$ trends towards higher wave numbers in both regularizations. The conservation diagnostics are shown in Figure~\ref{fig:AW_diagnostics}, where we again see an agreement with the properties discussed in Section~\ref{sec:properties_regularized_equations}, and the total variation of the density is also observed to remain consistent for each regularization.  We remark that the signs of the total entropy (LW-IGR) and total energy (LW-TIGRE) evolve differently here from their behaviour in the Sod shock tube test. This behaviour is due to the sign indefinite property of the information geometric correction present in both sets of equations.

\begin{figure}
    \centering
    \includegraphics[width=\linewidth]{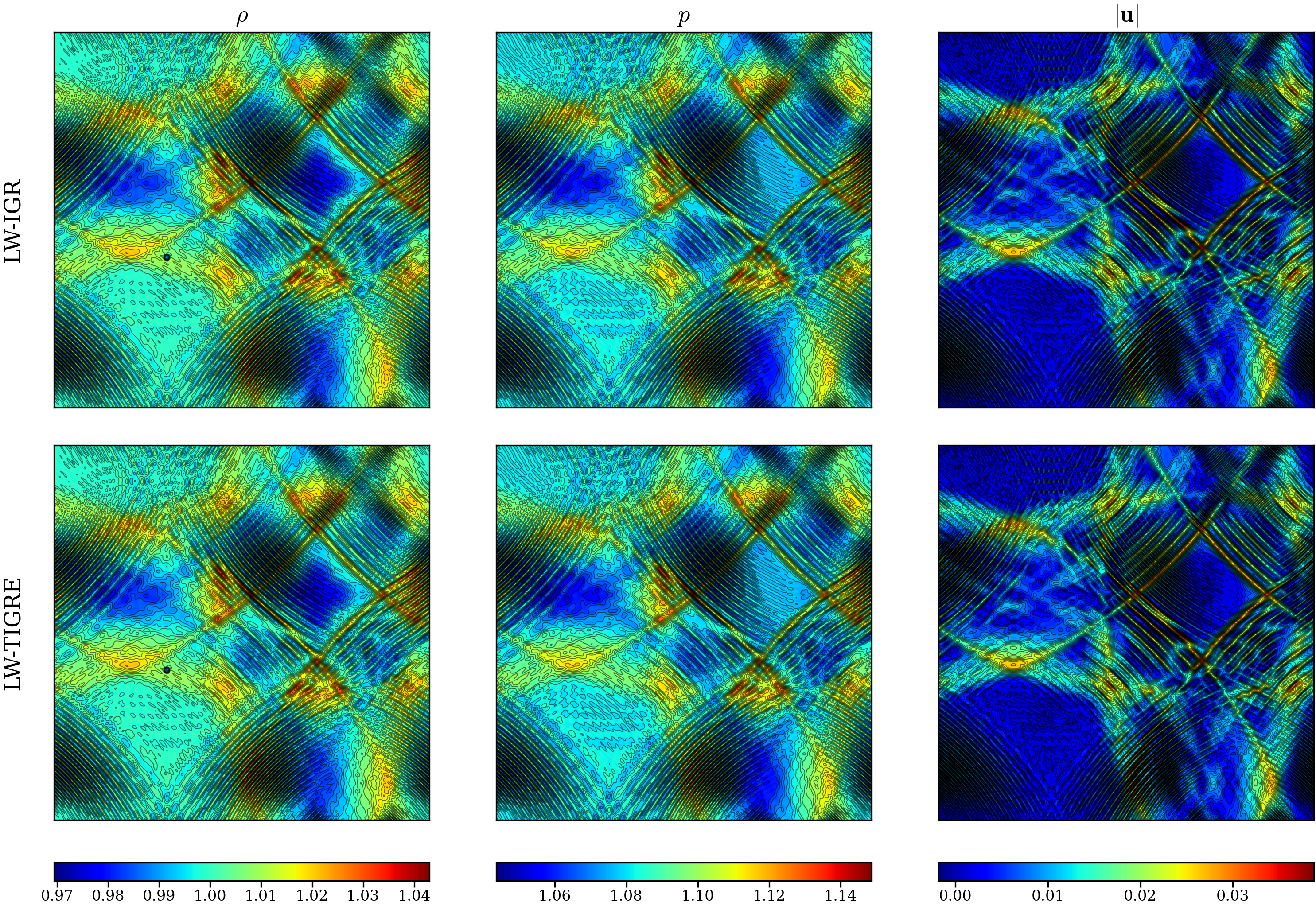}
    \vspace{0.1cm}
    \caption{Final solution profiles of the density, pressure, and magnitude of the velocity for the acoustic wave propagation test. The solutions profiles are observed to be nearly identical, illustrating how the acoustic wave propagation is consistent between both regularizations.}
    \label{fig:collision_front}
\end{figure}

\begin{figure}
    \centering
    \includegraphics[width=\linewidth]{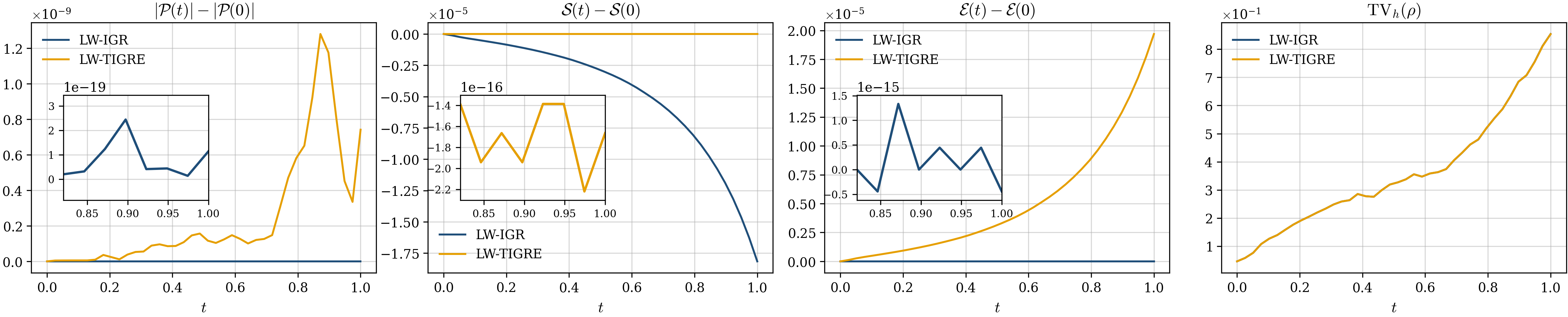}
    \caption{Diagnostics for the acoustic wave propagation test. We note that the sign of the total entropy for the LW-IGR and the total energy for the LW-TIGRE solution differ from the behaviour in the smoothed Sod shock tube test. This behaviour is due to the sign indefinite property of the information geometric correction for both sets of equations.}
    \label{fig:AW_diagnostics}
\end{figure}

\begin{figure}
    \centering
    \includegraphics[width=\linewidth]{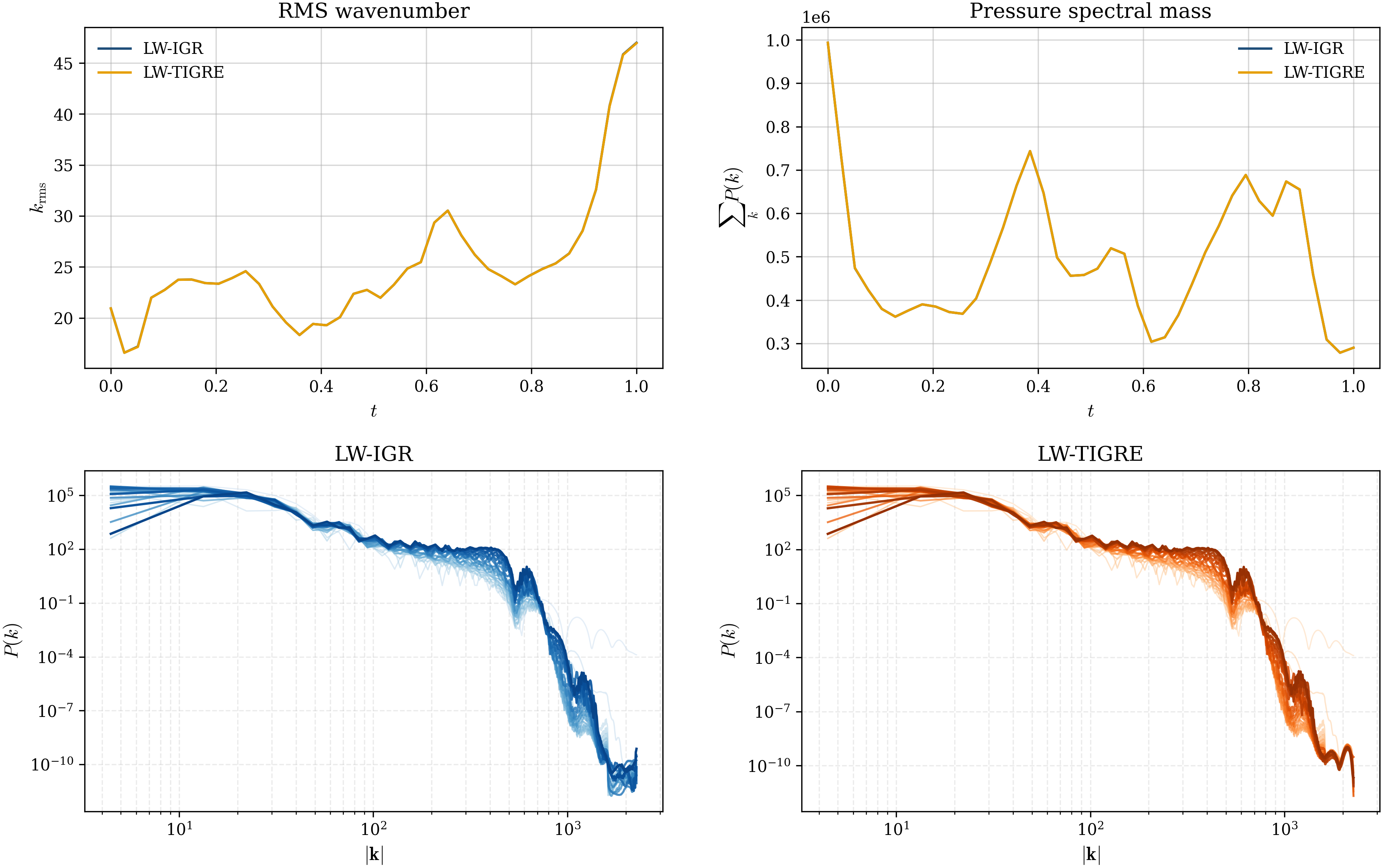}
    \caption{Spectral diagnostic comparison of the pressure field evolution for the LW-IGR and LW-TIGRE solutions. The mean wave number and total spectral mass (top row) of the pressure fields show similar behaviour for both regularizations, consistent with the evolution of the power spectrum (bottom row). The time evolution in the power spectrum is indicated by the increasing line opacity.}
    \label{fig:AW_spectral_diagnostics}
\end{figure}

\subsection{Kelvin--Helmholtz Instability}
\label{sec:KH_instability}

As a final numerical experiment, we consider a Kelvin--Helmholtz instability to assess the interaction of the regularization with vortex-dominated motion. The fluid is initialized using a shear layer with a small transverse perturbation to trigger a vortex-dominated instability. The initial velocity field profile is given by
\begin{equation*}
u_0(x,y) = \left(\frac{1}{4}\tanh\left(\frac{\delta y + w}{h}\right) - \frac{1}{4}\tanh\left(\frac{\delta y - w}{h}\right), \, A \sin(2\pi x) \exp\left(-\frac{\delta y^2}{h}\right)\right)\,,
\end{equation*}
with $\delta y = (y - y_0 + 0.5) \mod 1 - 0.5$ being a periodic distance from the shear layer centreline $y_0 = 0.5$, and the parameters were set to $w = 0.1$, $h = 0.01$, and $A = 0.01$. The density and entropy density are initialized to be uniform $\rho_0 = 1$ and $\pi_0 = 0.2$ such that no initial pressure perturbations are introduced. The dynamics are driven entirely by the shear instability, which sheds vortices at the scale of the initial perturbation.  \par

The simulation is performed with a grid $\Delta x = \Delta y = 1/1024$, and the final integration time is taken to be $t = 4$. The solution profiles at $t = 4$ are shown in Figure~\ref{fig:KH_instability}, where we observe a close similarity between the solutions for the pressure and the velocity field. The density in the LW-IGR, however, exhibits further oscillation along vortex filaments than for the LW-TIGRE solution, which can be seen as a difference in total variation indicated in Figure~\ref{fig:kh_diagnostics}. This discrepancy suggests a dynamical similarity with the Sod Shock tube test case where the cusp-singularity formed at pressure front collisions, now in combination with a vortex filament. In this case, the LW-TIGRE exhibits less dispersive error due to wave-vortex dynamics than in the LW-IGR solution, observed empirically in the right panel of figure~\ref{fig:kh_diagnostics}. The reduction in density wave oscillations near vortex filament in LW-TIGRE solution is a desirable property from the standpoint of stability of the numerical method and can be attributed to the additional regularization due to the $\chi$ potential. In Figure~\ref{fig:KH_instability_potentials}, we show the solution profiles for the potentials and observe balancing effects in the $\chi$ potential that possess an opposite sign to the $\Sigma$ potential. The conservative properties shown in Figure~\ref{fig:kh_diagnostics} indicate an agreement with the results of Section~\ref{sec:properties_regularized_equations}, where now the total entropy increases in the LW-IGR solution and the total energy decreases in the LW-TIGRE solution. We further assess the impacts on the regularization on the turbulent energy cascade for this solution in Figure~\ref{fig:kh_energy_cascade}, where we observe a similar decay structure in the kinetic energy spectrum. 
\begin{figure}[h!]
    \centering
    \includegraphics[width=\linewidth]{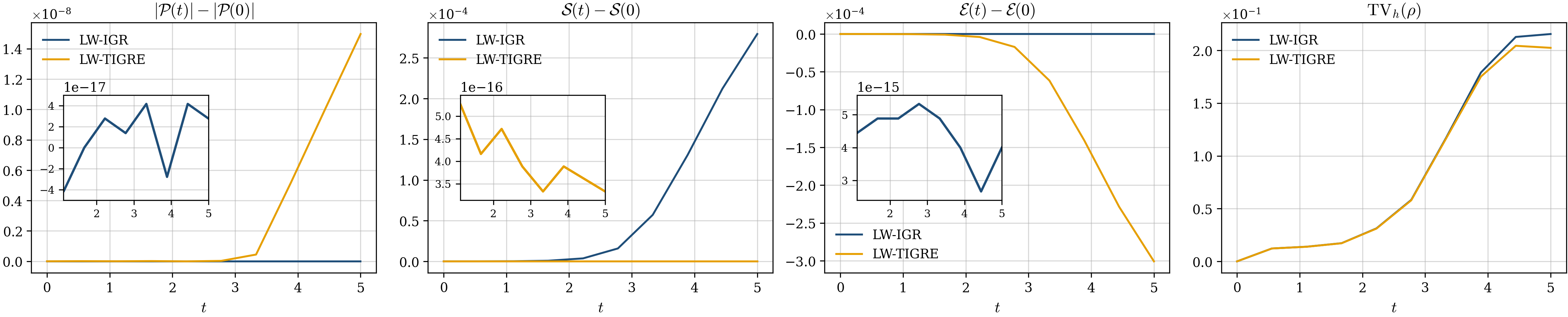}
    \caption{Conservation diagnostics for the Kelvin--Helmholtz instability test. The total variation of the mass densities differ, indicating an increase in dispersive error for the LW-IGR solution in comparison to the LW-TIGRE solution.}
    \label{fig:kh_diagnostics}
\end{figure}

\begin{figure}[h!]
    \centering
    \includegraphics[width=\linewidth]{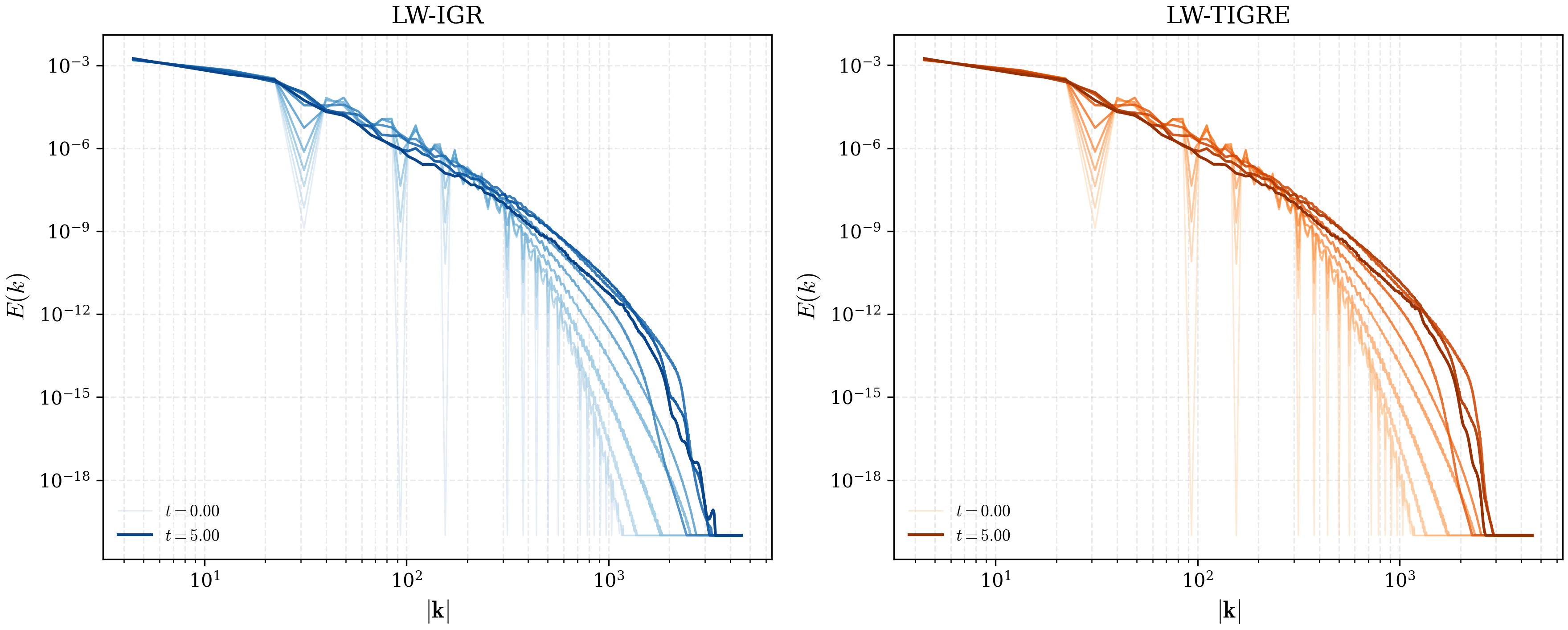}
    \caption{Energy cascades of both regularizations for the Kelvin--Helmholtz instability test.}
    \label{fig:kh_energy_cascade}
\end{figure}

\begin{figure}[h!]
    \centering
    \includegraphics[width = \linewidth]{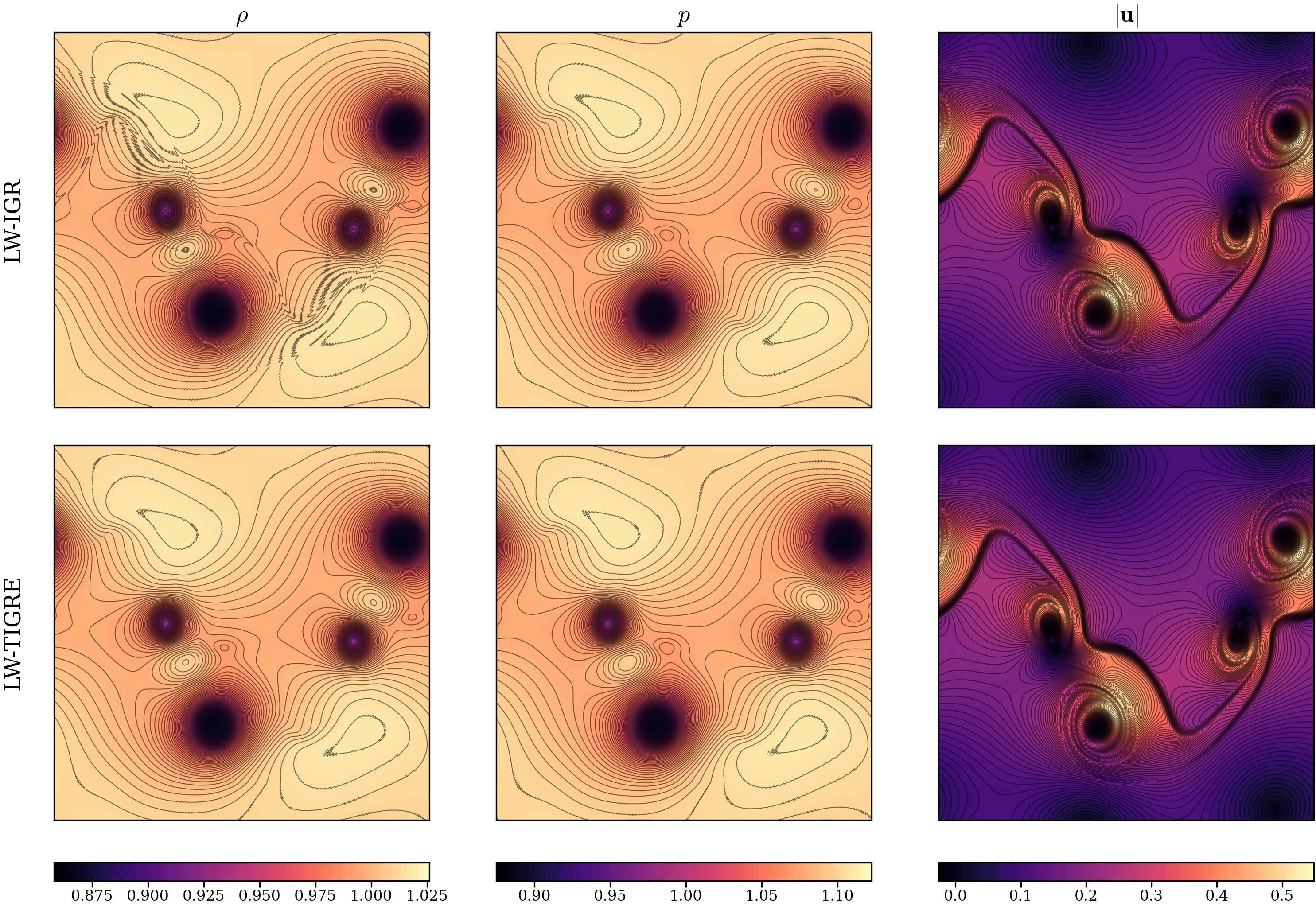}
    \vspace{0.1cm }
    \caption{Kelvin--Helmholtz instability test case at the final time step. We observe a damping of oscillatory behaviour in the interactions of the vortex motion and the density evolution in the TIGRE formulation in comparison to the IGR equations. The pressure and velocity fields however retain similar large scale features from one another.}
    \label{fig:KH_instability}
\end{figure}

\begin{figure}[h!]
    \centering
    \includegraphics[width = \linewidth]{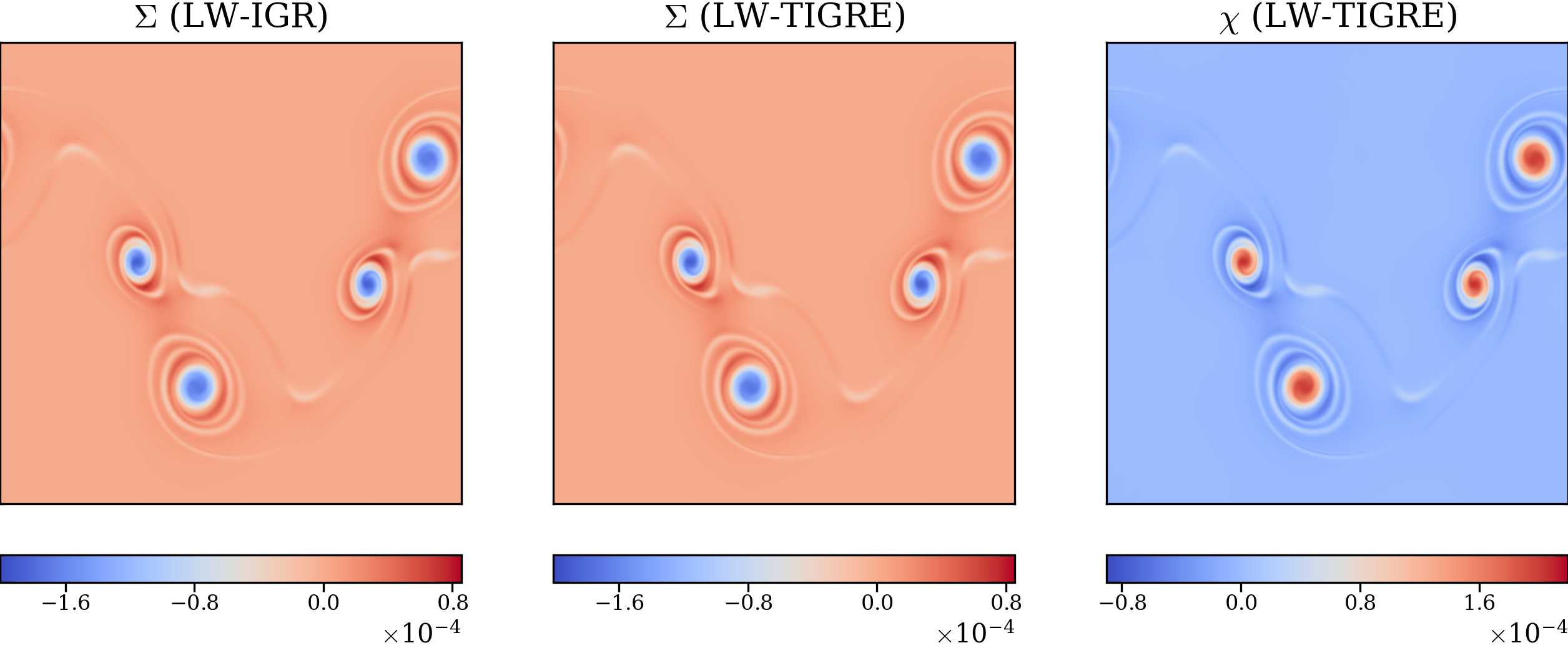}
    \vspace{0.1cm }
    \caption{The regularization potentials for the Kelvin--Helmholtz instability test case at the final time. The barotropic regularization $\Sigma$ for both models shows a similar dynamical behaviour. There is a sign difference between $\Sigma$ and $\chi$ for the thermodynamical regularization, indicating a balancing effect between the two potentials.}
    \label{fig:KH_instability_potentials}
\end{figure}

\section{Conclusion and Outlook}

In this work, we have proposed a thermodynamically constrained information geometric regularization (IGR) for compressible fluid flow. This represents an extension of the IGR for barotropic flows \cite{cao2023information} in two distinct ways. First, the new regularization incorporates a thermodynamic barrier defined by the relative entropy, modifying the original structure of the IGR equations with an additional regularization potential and entropic pressure contribution to the momentum equation. In doing so, we have resolved an open problem raised by Barham et al. \cite{barham2025hamiltonian} in regularizing the cusp-singularities observed in the IGR equations with a full thermodynamic state \eqref{full_compressible_euler_regularized}. These effects were validated numerically with experiments in one and two dimensions, where we have demonstrated the effects of the thermodynamic constraint in comparison with the previous IGR regularization. The constraint can be seen as selecting, among all regularized flows, the paths which are minimizing with respect to a thermodynamic length as defined by the Fisher--Rao metric. Second, we have provided a geometric formalism to derive the equations of motion using other barrier functions. This work offers a starting point for a number of lines of investigation and we conclude here with a few ideas that we hope could spur the development of new algorithms and theory at the confluence of infinite-dimensional geometry, mathematical physics, and scientific computing. 

\noindent
\emph{Numerical Solver and Benchmarking}
\vspace{0.1cm}

Developing high-performance numerical solvers for the TIGRE equations, similar to the high-order finite volume techniques for the IGR equations as described in \cite{wilfong2025simulating, radhakrishnan2026shocks}, would be important to benchmark the computational benefits and disadvantages of the approach in a performance-aware setting. The properties of the regularization demonstrated in this work motivate the design of other solver variants that incorporate alternative time-stepping routines and spatial discretizations along with their numerical analysis as interesting directions for future research. Integrating insights from research in hyperbolic conservation laws is also an important direction to motivate applications in other scientific domains that might benefit from this inviscid regularization strategy. Furthermore, investigating structure-preserving methods, such as those developed for finite elements \cite{gawlik2021variational, cotter2023compatible} and summation-by-parts finite difference methods \cite{svard2014review, bach2025sbp} within the context of information geometric mechanics might shed insight into approaches of capturing these regularizing properties at the discrete level. \\

\noindent
\emph{Stabilizing High-Order Methods}
\vspace{0.1cm}

There are a variety of compressible flow regimes that do not necessitate shock-capturing methods, such as those arising in atmospheric dynamics. Compressibility in the flow, however, still creates stiffness in the dynamics, and the use of high-order discretizations in these regimes is prone to instabilities; dispersive oscillations and uncontrolled aliasing error can violate fundamental constraints in the solution, triggering breakdown of the numerical solutions even when the dynamics are in a smooth regime. A variety of numerical techniques have been proposed to compensate for these instabilities such as the $2/3$ rule and filtering in pseudo-spectral methods \cite{boyd2001chebyshev, hou2007computing}, split-form and flux-form methods for spectral element and discontinuous Galerkin schemes \cite{kopriva2014energy, tadmor2016entropy}, and positivity-preserving limiters for finite volume schemes \cite{hu2013positivity}. Flux-differencing methods improve the robustness of high-order discretizations due to the enforcement of a discrete entropy inequality \cite{tadmor1987numerical, ranocha2023efficient, gassner2022stability}. An alternative approach to stabilization is to directly regularize the equations using a hyperviscosity \cite{boyd2001chebyshev}, entropy corrections \cite{abgrall2018general}, or artificial entropy diffusion \cite{chan2025artificial} for instance. Given the stabilizing effects of the information geometric regularization, an investigation into the use of other Hessian metric geometries as an alternative stabilizing technique for high-order numerical methods would be an interesting line of investigation. The design of barrier functions that are informed by the instabilities of the numerical method, damping spurious aliasing modes through the reformulation of the PDE model would be an interesting approach along this line.\\

\noindent
\emph{Thermodynamic Extensions}
\vspace{0.1cm}

In Section~\ref{sec:thermodynamic_constraint}, we motivated the use of the thermodynamic constraint through the relation to entropy production in the flow using an information geometric and a thermodynamic interpretation. While the relative entropy provides a natural and geometrically grounded starting point, extending the information geometric mechanics to more complex thermodynamic settings will require a better conceptualization based on the microscopic properties of gas dynamics and matter under consideration. In richer continuum setting such as reacting flows, multi-species plasmas, and systems undergoing phase transitions, the thermodynamic state space must be extended further and the reference equilibrium can become time-dependent. More fundamentally, near phase transitions the curvature as measured by the Fisher--Rao metric becomes an essential descriptor \cite{ruppeiner1995riemannian, ruppeiner2010thermodynamic}, suggesting curvature invariants of the statistical manifold may play a central role in correctly encoding the barrier structure in these regimes. Incorporating insights from geometric thermodynamics \cite{van2023thermodynamic, oikawa2025experimentally, zhong2024beyond, ito2024geometric} into these richer physical settings could offer a systematic approach to developing regularizations whose structure is more directly constrained by the geometry of the underlying thermodynamic phase space. \\

\noindent
\emph{Information Geometric Mechanics}
\vspace{0.1cm}

Our presentation has emphasized the interplay between Riemannian and information geometry underlying the equations of motions via the pullback dual geodesic formulation. Expanding on the information geometric mechanics \cite{cao2023information} will require further analytic and geometric formalism to investigate these mixed geometric structures. Leok and Zhang \cite{leok2017connecting} studied these connections in the finite-dimensional setting, where it was observed that information geometry could be viewed as an extension of geometric mechanics via Dirac mechanics on the Pontryagin bundle. A formalization of the Dirac mechanical structures, such as those studied by \cite{gay2026infinite}, on the diffeomorphism group and the extended thermodynamic state space could shed light into some more fundamental geometric structures underlying the information geometric mechanics.

\section*{\normalsize Acknowledgements}

The authors gratefully acknowledge that this research was supported in part by the Pacific Institute for the Mathematical Sciences (S.T. and R.J.S.), the Natural Sciences and Engineering Research Council of Canada under Discovery Grant RGPN-2020-04467 (R.J.S.), and Environment and Climate Change Canada.

\section*{\normalsize Statement of Author Contributions}

\textbf{Seth Taylor}: Conceptualization, Methodology, Formal Analysis, Software, Validation, Writing - Original Draft, Writing - Review \& Editing. \textbf{Raymond J. Spiteri}: Methodology, Validation, Writing - Original Draft, Writing - Review \& Editing, Supervision, Project administration, Funding acquisition. \textbf{St{\'e}phane Gaudreault}: Methodology, Validation, Writing - Review \& Editing, Supervision, Project administration, Funding acquisition. 

\small
\bibliographystyle{sn-mathphys}  
\bibliography{Bibliography}

@article{cao2023information,
  title={Information geometric regularization of the barotropic {Euler} equation},
  author={Cao, Ruijia and Sch{\"a}fer, Florian},
  journal={arXiv preprint arXiv:2308.14127},
  year={2023}
}

@article{khesin2021geometric,
  title={Geometric hydrodynamics and infinite-dimensional Newton’s equations},
  author={Khesin, Boris and Misio{\l}ek, Gerard and Modin, Klas},
  journal={Bulletin of the American Mathematical Society},
  volume={58},
  number={3},
  pages={377--442},
  year={2021}
}

@article{khesin2007shock,
  title={Shock waves for the {Burgers} equation and curvatures of diffeomorphism groups},
  author={Khesin, Boris and Misio{\l}ek, G},
  journal={Proceedings of the Steklov Institute of Mathematics},
  volume={259},
  number={1},
  pages={73--81},
  year={2007},
  publisher={Springer}
}

@article{chan2025artificial,
  title={An artificial viscosity approach to high order entropy stable discontinuous Galerkin methods},
  author={Chan, Jesse},
  journal={arXiv preprint arXiv:2501.16529},
  year={2025}
}

@article{hou2007computing,
  title={Computing nearly singular solutions using pseudo-spectral methods},
  author={Hou, Thomas Y and Li, Ruo},
  journal={Journal of Computational Physics},
  volume={226},
  number={1},
  pages={379--397},
  year={2007},
  publisher={Elsevier}
}

@book{boyd2001chebyshev,
  title     = {Chebyshev and Fourier spectral methods},
  author    = {Boyd, John P},
  year      = {2001},
  publisher = {Courier Corporation},
  address   = {Mineola, New York}
}

@article{kopriva2014energy,
  title={An energy stable discontinuous Galerkin spectral element discretization for variable coefficient advection problems},
  author={Kopriva, David A and Gassner, Gregor J},
  journal={SIAM Journal on Scientific Computing},
  volume={36},
  number={4},
  pages={A2076--A2099},
  year={2014},
  publisher={SIAM}
}

@incollection{tadmor2016entropy,
  author    = {Tadmor, Eitan},
  title     = {Entropy stable schemes},
  booktitle = {Handbook of Numerical Methods for Hyperbolic Problems},
  publisher = {Elsevier},
  address   = {Amsterdam},
  volume    = {17},
  pages     = {467--493},
  year      = {2016}
}

@article{holm1998euler,
  title={The {Euler}--{P}oincar{\'e} equations and semidirect products with applications to continuum theories},
  author={Holm, Darryl D and Marsden, Jerrold E and Ratiu, Tudor S},
  journal={Advances in Mathematics},
  volume={137},
  number={1},
  pages={1--81},
  year={1998},
  publisher={Elsevier}
}

@article{holm2009geometric,
  title={Geometric mechanics and symmetry: from finite to infinite dimensions},
  author={Holm, Darryl D and Schmah, Tanya and Stoica, Cristina},
  volume={12},
  year={2009},
  journal={Oxford University Press}
}

@article{marsden1997introduction,
  title={Introduction to mechanics and symmetry},
  author={Marsden, Jerrold E and Ratiu, Tudor S and Hermann, Robert},
  journal={SIAM Review},
  volume={39},
  number={1},
  pages={152--152},
  year={1997},
  publisher={Philadelphia, Society for Industrial and Applied Mathematics.}
}

@article{crooks2007measuring,
  title={Measuring thermodynamic length},
  author={Crooks, Gavin E},
  journal={Physical Review Letters},
  volume={99},
  number={10},
  pages={100602},
  year={2007},
  publisher={APS}
}

@article{bauer2024p,
  title={The L p-{Fisher}--{Rao} metric and {Amari}--{Chentsov} $\alpha$-Connections},
  author={Bauer, Martin and Le Brigant, Alice and Lu, Yuxiu and Maor, Cy},
  journal={Calculus of Variations and Partial Differential Equations},
  volume={63},
  number={2},
  pages={56},
  year={2024},
  publisher={Springer}
}

@article{hu2013positivity,
  title={Positivity-preserving method for high-order conservative schemes solving compressible {Euler} equations},
  author={Hu, Xiangyu Y and Adams, Nikolaus A and Shu, Chi-Wang},
  journal={Journal of Computational Physics},
  volume={242},
  pages={169--180},
  year={2013},
  publisher={Elsevier}
}

@article{abgrall2018general,
  title={A general framework to construct schemes satisfying additional conservation relations. Application to entropy conservative and entropy dissipative schemes},
  author={Abgrall, Remi},
  journal={Journal of Computational Physics},
  volume={372},
  pages={640--666},
  year={2018},
  publisher={Elsevier}
}

@article{tadmor1987numerical,
  title={The numerical viscosity of entropy stable schemes for systems of conservation laws. I},
  author={Tadmor, Eitan},
  journal={Mathematics of Computation},
  volume={49},
  number={179},
  pages={91--103},
  year={1987}
}

@article{ranocha2023efficient,
  title={Efficient implementation of modern entropy stable and kinetic energy preserving discontinuous Galerkin methods for conservation laws},
  author={Ranocha, Hendrik and Schlottke-Lakemper, Michael and Chan, Jesse and Rueda-Ram{\'\i}rez, Andr{\'e}s M and Winters, Andrew R and Hindenlang, Florian and Gassner, Gregor J},
  journal={ACM Transactions on Mathematical Software},
  volume={49},
  number={4},
  pages={1--30},
  year={2023},
  publisher={ACM New York, NY}
}

@article{gassner2022stability,
  title={Stability issues of entropy-stable and/or split-form high-order schemes: analysis of linear stability},
  author={Gassner, Gregor J and Sv{\"a}rd, Magnus and Hindenlang, Florian J},
  journal={Journal of Scientific Computing},
  volume={90},
  number={3},
  pages={79},
  year={2022},
  publisher={Springer}
}

@inproceedings{wilfong2025simulating,
  title={Simulating many-engine spacecraft: Exceeding 1 quadrillion degrees of freedom via information geometric regularization},
  author={Wilfong, Benjamin and Radhakrishnan, Anand and Le Berre, Henry and Vickers, Daniel and Prathi, Tanush and Tselepidis, Nikolaos and Dorschner, Benedikt and Budiardja, Reuben and Cornille, Brian and Abbott, Stephen and others},
  booktitle={Proceedings of the International Conference for High Performance Computing, Networking, Storage and Analysis},
  pages={14--24},
  year={2025}
}

@article{barham2025hamiltonian,
  title={Hamiltonian Information Geometric Regularization of the Compressible {Euler} Equations},
  author={Barham, William and Tran, Brian K and Southworth, Ben S and Sch{\"a}fer, Florian},
  journal={arXiv preprint arXiv:2512.13948},
  year={2025}
}

@article{amari2016information,
  title={Information geometry and its applications},
  author={Amari, Shun-ichi},
  volume={194},
  journal={Springer},
  year={2016}
}

@article{leveque1992numerical,
  title={Numerical methods for conservation laws},
  author={LeVeque, Randall J and Leveque, Randall J},
  volume={132},
  year={1992},
  journal={Springer}
}

@article{leok2017connecting,
  title={Connecting information geometry and geometric mechanics},
  author={Leok, Melvin and Zhang, Jun},
  journal={Entropy},
  volume={19},
  number={10},
  pages={518},
  year={2017},
  publisher={MDPI}
}

@article{bhagatwala2009modified,
  title={A modified artificial viscosity approach for compressible turbulence simulations},
  author={Bhagatwala, Ankit and Lele, Sanjiva K},
  journal={Journal of computational physics},
  volume={228},
  number={14},
  pages={4965--4969},
  year={2009},
  publisher={Elsevier}
}

@article{dolzhansky2012fundamentals,
  title={Fundamentals of geophysical hydrodynamics},
  author={Dolzhansky, Felix V},
  volume={103},
  year={2012},
  journal={Springer Science \& Business Media}
}

@article{lehner2014numerical,
  title={Numerical relativity and astrophysics},
  author={Lehner, Luis and Pretorius, Frans},
  journal={Annual Review of Astronomy and Astrophysics},
  volume={52},
  number={1},
  pages={661--694},
  year={2014},
  publisher={Annual Reviews}
}

@article{toro2013riemann,
  title={Riemann solvers and numerical methods for fluid dynamics: a practical introduction},
  author={Toro, Eleuterio F},
  year={2013},
  journal={Springer Science \& Business Media}
}

@article{godunov1959finite,
  title={Finite difference method for numerical computation of discontinuous solutions of the equations of fluid dynamics},
  author={Godunov, Sergei K and Bohachevsky, Ihor},
  journal={Matemati{\v{c}}eskij sbornik},
  volume={47},
  number={3},
  pages={271--306},
  year={1959}
}

@article{harten1983upstream,
  title={On upstream differencing and {Godunov}-type schemes for hyperbolic conservation laws},
  author={Harten, Amiram and Lax, Peter D and Leer, Bram van},
  journal={SIAM review},
  volume={25},
  number={1},
  pages={35--61},
  year={1983},
  publisher={SIAM}
}

@article{roe1981approximate,
  title={Approximate {Riemann} solvers, parameter vectors, and difference schemes},
  author={Roe, Philip L},
  journal={Journal of computational physics},
  volume={43},
  number={2},
  pages={357--372},
  year={1981},
  publisher={Elsevier}
}

@article{toro2009hll,
  title={The {HLL} and {HLLC} {Riemann} solvers},
  author={Toro, Eleuterio F},
  title={Riemann solvers and numerical methods for fluid dynamics: A practical introduction},
  pages={315--344},
  year={2009},
  journal={Springer}
}

@article{van1979towards,
  title={Towards the ultimate conservative difference scheme. {V}. {A} second-order sequel to {Godunov}'s method},
  author={Van Leer, Bram},
  journal={Journal of computational Physics},
  volume={32},
  number={1},
  pages={101--136},
  year={1979},
  publisher={Elsevier}
}

@article{harten1997uniformly,
  title={Uniformly high order accurate essentially non-oscillatory schemes, {III}},
  author={Harten, Ami and Engquist, Bjorn and Osher, Stanley and Chakravarthy, Sukumar R},
  journal={Journal of computational physics},
  volume={131},
  number={1},
  pages={3--47},
  year={1997},
  publisher={Elsevier}
}

@article{liu1994weighted,
  title={Weighted essentially non-oscillatory schemes},
  author={Liu, Xu-Dong and Osher, Stanley and Chan, Tony},
  journal={Journal of computational physics},
  volume={115},
  number={1},
  pages={200--212},
  year={1994},
  publisher={Elsevier}
}

@article{jiang1996efficient,
  title={Efficient implementation of weighted ENO schemes},
  author={Jiang, Guang-Shan and Shu, Chi-Wang},
  journal={Journal of computational physics},
  volume={126},
  number={1},
  pages={202--228},
  year={1996},
  publisher={Elsevier}
}

@article{vonneumann1950method,
  title={A method for the numerical calculation of hydrodynamic shocks},
  author={VonNeumann, John and Richtmyer, Robert D},
  journal={Journal of applied physics},
  volume={21},
  number={3},
  pages={232--237},
  year={1950},
  publisher={American Institute of Physics}
}

@inproceedings{persson2006sub,
  title={Sub-cell shock capturing for discontinuous {Galerkin} methods},
  author={Persson, Per-Olof and Peraire, Jaime},
  booktitle={44th AIAA aerospace sciences meeting and exhibit},
  pages={112},
  year={2006}
}

@article{nessyahu1990non,
  title={Non-oscillatory central differencing for hyperbolic conservation laws},
  author={Nessyahu, Haim and Tadmor, Eitan},
  journal={Journal of computational physics},
  volume={87},
  number={2},
  pages={408--463},
  year={1990},
  publisher={Elsevier}
}

@article{kurganov2000new,
  title={New high-resolution central schemes for nonlinear conservation laws and convection--diffusion equations},
  author={Kurganov, Alexander and Tadmor, Eitan},
  journal={Journal of computational physics},
  volume={160},
  number={1},
  pages={241--282},
  year={2000},
  publisher={Elsevier}
}

@article{pirozzoli2011numerical,
  title={Numerical methods for high-speed flows},
  author={Pirozzoli, Sergio},
  journal={Annual review of fluid mechanics},
  volume={43},
  number={1},
  pages={163--194},
  year={2011},
  publisher={Annual Reviews}
}

@article{wilkins1980use,
  title={Use of artificial viscosity in multidimensional fluid dynamic calculations},
  author={Wilkins, Mark L},
  journal={Journal of computational physics},
  volume={36},
  number={3},
  pages={281--303},
  year={1980},
  publisher={Elsevier}
}

@article{frisch2008hyperviscosity,
  title={Hyperviscosity, {Galerkin} truncation, and bottlenecks in turbulence},
  author={Frisch, Uriel and Kurien, Susan and Pandit, Rahul and Pauls, Walter and Ray, Samriddhi Sankar and Wirth, Achim and Zhu, Jian-Zhou},
  journal={Physical review letters},
  volume={101},
  number={14},
  pages={144501},
  year={2008},
  publisher={APS}
}

@article{cook2005hyperviscosity,
  title={Hyperviscosity for shock-turbulence interactions},
  author={Cook, Andrew W and Cabot, William H},
  journal={Journal of Computational Physics},
  volume={203},
  number={2},
  pages={379--385},
  year={2005},
  publisher={Elsevier}
}

@article{cao2024information,
  title={Information geometric regularization of unidimensional pressureless {Euler} equations yields global strong solutions},
  author={Cao, Ruijia and Sch{\"a}fer, Florian},
  journal={arXiv preprint arXiv:2411.15121},
  year={2024}
}

@article{bauer2016uniqueness,
  title={Uniqueness of the {Fisher--Rao} metric on the space of smooth densities},
  author={Bauer, Martin and Bruveris, Martins and Michor, Peter W},
  journal={Bulletin of the London Mathematical Society},
  volume={48},
  number={3},
  pages={499--506},
  year={2016},
  publisher={Oxford University Press}
}

@article{bauer2015quiet,
  title={The quiet revolution of numerical weather prediction},
  author={Bauer, Peter and Thorpe, Alan and Brunet, Gilbert},
  journal={Nature},
  volume={525},
  number={7567},
  pages={47--55},
  year={2015},
  publisher={Nature Publishing Group UK London}
}

@article{anderson1995computational,
  title={Computational fluid dynamics},
  author={Anderson, John David and Wendt, John and others},
  volume={206},
  year={1995},
  journal={Springer}
}

@article{ruppeiner2010thermodynamic,
  title={Thermodynamic curvature measures interactions},
  author={Ruppeiner, George},
  journal={American Journal of Physics},
  volume={78},
  number={11},
  pages={1170--1180},
  year={2010},
  publisher={AIP Publishing}
}

@article{ay2017information,
  title={Information geometry},
  author={Ay, Nihat and Jost, J{\"u}rgen and V{\^a}n L{\^e}, H{\^o}ng and Schwachh{\"o}fer, Lorenz},
  volume={64},
  year={2017},
  journal={Springer}
}

@article{khesin2024information,
  title={Information geometry of diffeomorphism groups},
  author={Khesin, Boris and Misio{\l}ek, Gerard and Modin, Klas},
  journal={arXiv preprint arXiv:2411.03265},
  year={2024}
}

@article{lauritzen1987statistical,
  title={Statistical manifolds},
  author={Lauritzen, Stefan L},
  journal={Differential geometry in statistical inference},
  volume={10},
  number={2},
  pages={163--216},
  year={1987},
  publisher={Institute of Mathematical Statistics Hayward}
}

@article{rao1945information,
  title={Information and the accuracy attainable in the estimation of statistical parameters},
  author={Rao, C Radhakrishna and others},
  journal={Bull. Calcutta Math. Soc},
  volume={37},
  number={3},
  pages={81--91},
  year={1945}
}

@article{chentsov1982statistical,
  title={Statistical decision rules and optimal inference},
  author={Chentsov, Nikolai Nikolaevich},
  year={1982},
  journal={American Mathematical Society}
}

@article{amari1982differential,
  title={Differential geometry of curved exponential families-curvatures and information loss},
  author={Amari, Shun-Ichi},
  journal={The Annals of Statistics},
  pages={357--385},
  year={1982},
  publisher={JSTOR}
}

@article{pistone1995infinite,
  title={An infinite-dimensional geometric structure on the space of all the probability measures equivalent to a given one},
  author={Pistone, Giovanni and Sempi, Carlo},
  journal={The annals of statistics},
  pages={1543--1561},
  year={1995},
  publisher={JSTOR}
}

@article{shima1997geometry,
  title={Geometry of Hessian manifolds},
  author={Shima, Hirohiko and Yagi, Katsumi},
  journal={Differential geometry and its applications},
  volume={7},
  number={3},
  pages={277--290},
  year={1997},
  publisher={Elsevier}
}

@article{feng2009far,
  title={Far-from-equilibrium measurements of thermodynamic length},
  author={Feng, Edward H and Crooks, Gavin E},
  journal={Physical Review E—Statistical, Nonlinear, and Soft Matter Physics},
  volume={79},
  number={1},
  pages={012104},
  year={2009},
  publisher={APS}
}

@article{amari1980theory,
  title={Theory of information spaces: A differential geometrical foundation of statistics},
  author={Amari, Shun-ichi},
  journal={Post RAAG Reports},
  year={1980}
}

@article{otto2001geometry,
  title={The geometry of dissipative evolution equations: the porous medium equation},
  author={Otto, Felix},
  year={2001},
  publisher={Taylor \& Francis}
}

@article{benamou2000computational,
  title={A computational fluid mechanics solution to the Monge-Kantorovich mass transfer problem},
  author={Benamou, Jean-David and Brenier, Yann},
  journal={Numerische Mathematik},
  volume={84},
  number={3},
  pages={375--393},
  year={2000},
  publisher={Springer-Verlag Berlin/Heidelberg}
}

@article{lott2009ricci,
  title={Ricci curvature for metric-measure spaces via optimal transport},
  author={Lott, John and Villani, C{\'e}dric},
  journal={Annals of Mathematics},
  pages={903--991},
  year={2009},
  publisher={JSTOR}
}

@article{li2021hessian,
  title={Hessian metric via transport information geometry},
  author={Li, Wuchen},
  journal={Journal of Mathematical Physics},
  volume={62},
  number={3},
  year={2021},
  publisher={AIP Publishing}
}

@article{radhakrishnan2026shocks,
  title={Shocks without shock capturing: Information geometric regularization of finite volume methods for Navier--Stokes-like problems},
  author={Radhakrishnan, Anand and Wilfong, Benjamin and Bryngelson, Spencer H and Sch{\"a}fer, Florian},
  journal={arXiv preprint arXiv:2604.06546},
  year={2026}
}

@article{gawlik2021variational,
  title={A variational finite element discretization of compressible flow},
  author={Gawlik, Evan S and Gay-Balmaz, Fran{\c{c}}ois},
  journal={Foundations of Computational Mathematics},
  volume={21},
  number={4},
  pages={961--1001},
  year={2021},
  publisher={Springer}
}

@article{cotter2023compatible,
  title={Compatible finite element methods for geophysical fluid dynamics},
  author={Cotter, Colin J},
  journal={Acta Numerica},
  volume={32},
  pages={291--393},
  year={2023},
  publisher={Cambridge University Press}
}

@article{svard2014review,
  title={Review of summation-by-parts schemes for initial--boundary-value problems},
  author={Sv{\"a}rd, Magnus and Nordstr{\"o}m, Jan},
  journal={Journal of Computational Physics},
  volume={268},
  pages={17--38},
  year={2014},
  publisher={Elsevier}
}

@article{bach2025sbp,
  title={{SBP-FDEC}: {Summation-by-Parts Finite Difference Exterior Calculus}},
  author={Bach, Daniel and Rueda-Ram{\'\i}rez, Andr{\'e}s M and Sonnendr{\"u}cker, Eric and Fern{\'a}ndez, David C and Gassner, Gregor J},
  journal={arXiv preprint arXiv:2511.20529},
  year={2025}
}

@article{ito2024geometric,
  title={Geometric thermodynamics for the {Fokker--Planck} equation: stochastic thermodynamic links between information geometry and optimal transport},
  author={Ito, Sosuke},
  journal={Information geometry},
  volume={7},
  pages={441--483},
  year={2024},
  publisher={Springer}
}

@article{van2023thermodynamic,
  title={Thermodynamic unification of optimal transport: Thermodynamic uncertainty relation, minimum dissipation, and thermodynamic speed limits},
  author={Van Vu, Tan and Saito, Keiji},
  journal={Physical Review X},
  volume={13},
  number={1},
  pages={011013},
  year={2023},
  publisher={APS}
}

@article{seifert2012stochastic,
  title={Stochastic thermodynamics, fluctuation theorems and molecular machines},
  author={Seifert, Udo},
  journal={Reports on progress in physics},
  volume={75},
  number={12},
  pages={126001},
  year={2012},
  publisher={IOP Publishing}
}

@article{oikawa2025experimentally,
  title={Experimentally achieving minimal dissipation via thermodynamically optimal transport},
  author={Oikawa, Shingo and Nakayama, Yohei and Ito, Sosuke and Sagawa, Takahiro and Toyabe, Shoichi},
  journal={Nature Communications},
  volume={16},
  number={1},
  pages={10424},
  year={2025},
  publisher={Nature Publishing Group UK London}
}

@article{gay2026infinite,
  title={Infinite-dimensional {Lagrange--Dirac} systems with boundary energy flow II: {Field} theories with bundle-valued forms},
  author={Gay--Balmaz, Fran{\c{c}}ois and Abella, {\'A}lvaro Rodr{\'\i}guez and Yoshimura, Hiroaki},
  journal={Journal of Geometry and Physics},
  pages={105854},
  year={2026},
  publisher={Elsevier}
}

@article{zhong2024beyond,
  title={Beyond linear response: Equivalence between thermodynamic geometry and optimal transport},
  author={Zhong, Adrianne and DeWeese, Michael R},
  journal={Physical Review Letters},
  volume={133},
  number={5},
  pages={057102},
  year={2024},
  publisher={APS}
}

@article{ruppeiner1995riemannian,
  title={Riemannian geometry in thermodynamic fluctuation theory},
  author={Ruppeiner, George},
  journal={Reviews of Modern Physics},
  volume={67},
  number={3},
  pages={605},
  year={1995},
  publisher={APS}
}

\end{document}